\documentclass[smallcondensed, numbook, envcountsame]{svjour3}

\smartqed

\usepackage[utf8]{inputenc}
\usepackage[T1]{fontenc}

\usepackage{amsmath, amssymb}
\usepackage{mathrsfs}
\usepackage{enumerate}
\usepackage{url}
\usepackage{graphicx}
\usepackage[all,cmtip]{xy}
\usepackage{relsize}

\newcommand{\eps}{\varepsilon}
\newcommand{\setsuch}[2]{\left\{ #1 \; \middle| \; #2 \right\}}
\newcommand{\restr}[2]{{\left. #1 \right|}_{#2}}

\newcommand{\subl}{\mathsmaller{<}}
\newcommand{\subg}{\mathsmaller{>}}
\newcommand{\sube}{\mathsmaller{=}}
\newcommand{\suble}{\mathsmaller{\leq}}
\newcommand{\subge}{\mathsmaller{\geq}}

\DeclareMathOperator{\Id}{Id}
\DeclareMathOperator{\diam}{diam}

\newcommand{\halfperp}{\mathsf{L}}

\newcommand{\ie}{i.e.\ }

\newcommand{\longlongrightarrow}{\xrightarrow{\hspace*{1cm}}}

\begin{document}

\title{Fundamental domains for properly discontinuous affine groups}
\author{Ilia Smilga}
\institute{I. Smilga \at
              Département de Mathématiques, Université Paris-Sud 11, F-91405 Orsay Cedex, France \\
              \email{ilia.smilga@u-psud.fr}}
\date{Received: date / Accepted: date}

\maketitle

\begin{abstract}
We construct a fundamental region for the action on the $2d+1$-dimensional affine space of some free, discrete, properly discontinuous groups of affine transformations preserving a quadratic form of signature $(d+1, d)$, where $d$ is any odd positive integer.
\end{abstract}

\section{Introduction}
\label{sec:intro}

\subsection{Background and motivation}
\label{sec:background}

The present paper is part of a larger effort to understand discrete groups $\Gamma$ of affine transformations (subgroups of the affine group $GL_n(\mathbb{R}) \rtimes \mathbb{R}^n$) acting properly discontinuously on the affine space $\mathbb{R}^n$. The case where $\Gamma$ consists of isometries (in other words, $\Gamma \subset O_n(\mathbb{R}) \rtimes \mathbb{R}^n$) is well-understood: a classical theorem by Bieberbach says that such a group always has an abelian subgroup of finite index.

Define a \emph{crystallographic} group to be a discrete group $\Gamma \subset GL_n(\mathbb{R}) \rtimes \mathbb{R}^n$ acting properly discontinuously and such that the quotient space $\mathbb{R}^n / \Gamma$ is compact. In \cite{Aus64}, Auslander conjectured that any crystallographic group is virtually solvable, that is, contains a solvable subgroup of finite index. Later, Milnor \cite{Mil77} asked whether this statement is actually true for any affine group acting properly discontinuously. The answer turned out to be negative: Margulis \cite{Mar83,Mar87} gave a counterexample in dimension 3. On the other hand, Fried and Goldman \cite{FG83} proved the Auslander conjecture in dimension 3 (the cases $n=1$ and $2$ are easy). Later, Abels, Margulis and Soifer proved it in dimension $n \leq 6$. See \cite{AbSur} for a survey of already known results.

In his PhD thesis and subsequent papers \cite{Dru92,Dru93}, Drumm elaborated on Margulis's result by explicitly describing fundamental domains for the groups $\Gamma$ introduced by Margulis, which allowed him in particular to deduce the topology of the quotient $\mathbb{R}^3 / \Gamma$. On the other hand, Abels, Margulis and Soifer \cite{AMS02} constructed a family of counterexamples to Milnor's conjecture in dimension $4n+3$, preserving a quadratic form of signature $(2n+2,2n+1)$. The purpose of this paper is to adapt Drumm's construction to Abels-Margulis-Soifer groups: describe a fundamental domain and deduce the topology of the quotient space. Here is the main result:

\begin{theorem}[Main Theorem]
Let $d$ be an odd positive integer. Then any generalized Schottky subgroup of $SO(d+1, d)$ with sufficiently contracting generators has a nonempty open set of affine deformations $\Gamma$ that act properly discontinuously on $\mathbb{R}^{d+1, d}$, with the quotient $\mathbb{R}^{d+1, d}/\Gamma$ homeomorphic to a solid $(2d+1)$-dimensional handlebody.  
\end{theorem}

To do this, we use mainly two sources of inspiration. The first one is of course \cite{AMS02}, the original work of Abels, Margulis and Soifer. The second one is an article by Charette and Goldman \cite{CG00} presenting Drumm's results.

\subsection{Plan of the paper}
\label{sec:plan}

We start, in section \ref{sec:basic}, by giving some elementary geometrical properties of a space equipped with a form of signature $(d+1, d)$ where $d$ is odd. We describe, in subsection \ref{sec:MTIS'es}, its maximal totally isotropic subspaces; in subsection \ref{sec:frames}, its pseudohyperbolic maps (roughly maps whose space of fixed points has the smallest possible dimension); in subsection \ref{sec:orientation}, an orientation trick (taken from \cite{AMS02}) that allows to extend any two transversal maximal totally isotropic subspaces into half-$d+1$-dimensional spaces that still have zero intersection. Finally, in subsection \ref{sec:metric}, we introduce metrics on various spaces (in particular projective spaces) we need to work with, and we define the strength of contraction of a pseudohyperbolic map.

In the next two sections, we consider subgroups of $SO(d+1,d)$ generated by pseudohyperbolic maps. In section  \ref{sec:group}, we study their action on $\mathbb{P}(\Lambda^d \mathbb{R}^{d+1,d})$. We show that, provided the generators are sufficiently contracting, such a group is free and every element is pseudohyperbolic. We also control the geometry and strength of contraction of all cyclically reduced words on the generators. This result is very similar to Lemma 5.24 from \cite{AMS02}, and we follow closely its proof. (For a more concise proof of a similar result, see also section 6 of \cite{Ben96}.)

In section \ref{sec:tennis_ball}, we study the action of these subgroups directly on $\mathbb{P}(\mathbb{R}^{d+1,d})$. We show that, supposing again that the generators are sufficiently contracting, this action is similar to the action of a Schottky group (which shows again that such a group is free). The way we construct the fundamental domain was partly inspired by Drumm's ideas, but his "crooked planes" do not directly generalize to higher dimensions. Instead, we have used "angular" neighborhoods of some half-spaces (namely of the "positive wings" defined in section \ref{sec:orientation}).

Finally, in section \ref{sec:affine}, we study affine groups $\Gamma$ whose linear parts satisfy the conditions of the previous two sections. We prove the Main Theorem (after stating it more precisely: see Theorem \ref{main_theorem}). Here we closely follow section 4 of \cite{CG00}. First, we describe a set $\mathcal{H}^0$ as the complement to $2n$ "sources" and "sinks" corresponding to the $n$ generators of $\Gamma$. We show (Proposition \ref{fundamental_region}) that under some conditions, $\mathcal{H}^0$ is a fundamental domain for $\Gamma$. Indeed, we see immediately that its images under elements of the group "fit together nicely". To prove that they cover the whole space, by contradiction, we turn our attention to a hypothetical point not covered by any "tile". We include it in a nested sequence of domains, then show (by methods adapted from \cite{CG00}) that these domains must, in a sense, run away to infinity.

\section{Conventions, definitions and basic properties}
\label{sec:basic}

Let $p$ and $q$ be two positive integers. We write $\mathbb{R}^{p, q}$ as shorthand for the space $\mathbb{R}^{p+q}$ equipped with a quadratic form $Q$ of signature $(p, q)$. The group of automorphisms of $\mathbb{R}^{p, q}$ (that is, automorphisms of $\mathbb{R}^{p+q}$ that preserve the quadratic form) is $O(p, q)$. This group has four connected components; we call $SO^+(p,q)$ the connected component of the identity.

We equip $\mathbb{R}^{p, q}$ with some additional structure. We choose a maximal positive definite subspace $S$ of $\mathbb{R}^{p, q}$, and we set $T = S^\perp$ the corresponding maximal negative definite subspace. We may then define orthogonal projections $\pi_S: \mathbb{R}^{p, q} \to S$ and $\pi_T: \mathbb{R}^{p, q} \to T$, and positive definite forms $N_S := \restr{Q}{S}$ and $N_T := -\restr{Q}{T}$, so that
\begin{equation}
\label{eq:form_decomposition}
\forall x \in \mathbb{R}^{p,q},\quad Q(x) = N_S(\pi_S(x)) - N_T(\pi_T(x)).
\end{equation}

\subsection{Maximal totally isotropic subspaces}
\label{sec:MTIS'es}

From now on, the acronym MTIS stands for a maximal totally isotropic subspace. If $V$ is a MTIS of $\mathbb{R}^{p, q}$, then (supposing that $p \geq q$) we have $\dim V = q$, $V \subset V^\perp$ and $\dim V^\perp = p$. We write $\mathscr{L}$ the set of all MTIS'es.

A very useful tool for the study of MTIS'es is the following bijection between $\mathscr{L}$ and the space $O(T, S)$ of orthogonal linear maps from $T$ to $S$ (seen as Euclidean spaces via the forms $N_S$ and $N_T$):
\begin{equation}
\label{eq:MTIS_bijection}
\xymatrix@R=3pt{
  \mathscr{L}                            \ar@{<->}[r]^{\sim} & O(T, S) \\
  V                                      \ar@{|->}[r]        & f_V := \pi_S \circ \left(\restr{\pi_T}{V}\right)^{-1} \\
  V_f := \setsuch{t + f(t)}{t \in T} & f \ar@{|->}[l]
}
\end{equation}

It is straightforward to check that both of these maps are well-defined and reciprocal to each other. Indeed, for any $V \in \mathscr{L}$ and $f \in O(T, S)$, we have:
\begin{itemize}
\item $\restr{\pi_T}{V}$ is bijective. Indeed, since $V \cap T^\perp = V \cap S = \emptyset$, this map is injective, and the spaces $T$ and $V$ have equal dimension.
\item $f_V \in O(T, S)$. Indeed, let $t \in T$; we set $v := (\restr{\pi_T}{V})^{-1}(t)$. Then $v \in V$, and we have $0 = Q(v) = N_S(\pi_S(v)) - N_T(\pi_T(v)) = N_S(f_V(t)) - N_T(t)$.
\item $V_f$ is a MTIS. Indeed, this space has dimension $q$, and for all $t \in T$, we have $Q(t + f(t)) = N_S(f(t)) - N_T(t) = 0$.
\item $V_{f_V} = V$. Indeed, let $v \in V$; then we have $v = \pi_T(v) + \pi_S(v) = \pi_T(v) + f_V(\pi_T(v))$, hence $v \in V_{f_V}$; and we know that $V$ and $V_{f_V}$ have the same dimension.
\item $f_{V_f} = f$. Indeed, let $t \in T$; then we have $f_{V_f}(t) = \pi_S(t + f(t)) = f(t)$. 
\end{itemize}

Here is a first application of this bijection. Later in the paper we shall prove some facts about families of $2n$ pairwise transversal MTIS'es. It would be wise to check that these statements are not vacuous, \ie that such families do indeed exist. This might seem obvious, but it turns out that, while it works for the particular values of $p$ and $q$ we deal with, it is false in general:

\begin{lemma}
\label{pairwise_transversal}
Let $p$, $q$ be two integers, $p \geq q \geq 0$. Then it is possible to find infinitely many pairwise transversal MTIS'es in $\mathbb{R}^{p, q}$, {\bf unless} $p = q$ and $p$ is odd, in which case it is impossible to find more than two of them.
\end{lemma}
\begin{proof}
Let $V_1$ and $V_2$ be two MTIS'es, and $f_i := f_{V_i}$ their images under the bijection \eqref{eq:MTIS_bijection}. We claim that $V_1$ and $V_2$ are transversal iff $f_1 - f_2$ is injective. Indeed, we have
\[ x \in V_1 \cap V_2 \iff \exists t \in T,\quad x = t + f_1(t) = t + f_2(t), \]
hence $V_1 \cap V_2 = 0 \iff \ker (f_1 - f_2) = 0$.

The question now becomes: how many orthogonal maps from $T$ to $S$ --- or, equivalently, from $\mathbb{R}^q$ to $\mathbb{R}^p$ --- can we find such that their differences are pairwise injective, \ie such that the images of any nonzero vector under these maps are pairwise different?

Suppose first that we may find an even integer $r$ such that $q \leq r \leq p$. Let $f_0: \mathbb{R}^q \to \mathbb{R}^p$ be any orthogonal (hence injective) map, and let $E$ be any $r$-dimensional linear space such that $f_0(\mathbb{R}^q) \subset E \subset \mathbb{R}^p$. Then we may find in $O(E)$ an infinite subgroup whose nontrivial elements have no fixed points: for example, the group $G$ formed by matrices
\[\begin{pmatrix}
R_\theta & 0      & 0 \\
0        & \ddots & 0 \\
0        & 0      & R_\theta
\end{pmatrix}\]
(where $R_\theta = \begin{pmatrix}\cos \theta & - \sin \theta \\ \sin \theta & \cos \theta\end{pmatrix}$), with $\theta$ running in $\mathbb{R}$. Now consider the set of all maps $g \circ f_0$ with $g \in G$. Let $x \in \mathbb{R}^q \setminus \{0\}$: then $f_0(x) \neq 0$, and the images of $f_0(x)$ under the elements of $G$ are pairwise different. It follows that these maps have indeed pairwise injective differences.

Otherwise, we have $p = q$ and $p$ is odd. The identity and the map $x \mapsto -x$ are two maps of $O(p)$ with injective difference. Now take any three maps in $O(p)$. Then at least two of them, let us call them $f_1$ and $f_2$, have the same determinant: in other terms $f_1 \circ f_2^{-1} \in SO(p)$. But for odd $p$, any map of $SO(p)$ has a fixed point. It follows that $f_1 - f_2$ is not injective.
\qed
\end{proof}

\subsection{Pseudohyperbolic maps and frames}
\label{sec:frames}

From now on, we fix a positive integer $d$ and we set $(p, q) = (d+1, d)$. Take any map $g \in GL(\mathbb{R}^{d+1, d})$. Then we may decompose $\mathbb{R}^{d+1, d}$ into a direct sum of three spaces $\mathbb{R}^{d+1, d} = V_{\subl}(g) \oplus V_{\sube}(g) \oplus V_{\subg}(g)$ stable by $g$ and such that all eigenvalues $\lambda$ of $\restr{g}{V_{\subl}(g)}$ (resp. $V_{\sube}$, $V_{\subg}$) satisfy $|\lambda| < 1$ (resp. $=1$, $>1$).
\begin{definition}
We shall say that $g$ is \emph{pseudohyperbolic} if $g \in O(d+1, d)$, $\dim V_{\sube}(g) = 1$ and the eigenvalue of $g$ lying in $V_{\sube}(g)$ is $1$ (not $-1$). (As we will soon show, all pseudohyperbolic maps actually lie in $SO(d+1, d)$). In this case, we define the \emph{frame} of $g$ to be the ordered pair $\mathcal{V}(g) := (V_{\subl}(g), V_{\subg}(g))$, and the \emph{dynamical part} of $g$ (as opposed to the frame, which is the "geometrical part") to be the map $g_{\subl} := \restr{g}{V_{\subl}(g)}$.
\end{definition}
Then a pseudohyperbolic map is uniquely defined by its frame and dynamical part. However, these must satisfy some conditions. To state them, we shall need the following notation: for any linear map $g$, we denote by $\rho(g)$ its \emph{spectral radius}, that is, the largest modulus of any eigenvalue of $g$.
\begin{proposition}
\label{frame_and_dynamical_part}
Pseudohyperbolic maps are in one-to-one correspondence (via the previous definition) with (ordered) triples $(V_{\subl}, V_{\subg}, g_{\subl})$ such that $V_{\subl}$ and $V_{\subg}$ are two transversal MTIS'es and $g_{\subl}$ is an automorphism of $V_{\subl}$ with $\rho(g_{\subl}) < 1$.
\end{proposition}
\begin{proof}
First, let us check that the frame and dynamical part of any pseudohyperbolic map $g$ do satisfy the required conditions. Indeed:
\begin{itemize}
\item The fact that $g_{\subl}$ is an automorphism of $V_{\subl}$ and the limitation on its spectral radius follow immediately from the definition of $V_{\subl}$.
\item Also by definition, $V_{\subl}(g) \cap V_{\subg}(g) = 0$.
\item Let $x_{\subl} \in V_{\subl}(g)$. Then we have
\[Q(x_{\subl})
= \lim_{n \to +\infty} Q(g^n(x_{\subl}))
= Q \left( \lim_{n \to +\infty} g_{\subl}^n(x_{\subl}) \right)
= 0,\]
since $\rho(g_{\subl}) < 1$. This shows that $V_{\subl}$ is a totally isotropic subspace.
\item Similarly, by using $g^{-1}$ instead of $g$, we can show that $V_{\subg}$ is totally isotropic. Now since $\mathbb{R}^{d+1,d} = V_{\subl} \oplus V_{\sube} \oplus V_{\subg}$, we have
\[2d+1 = \dim V_{\subl} + \dim V_{\sube} + \dim V_{\subg} \leq d + 1 + d = 2d+1,\]
hence the inequality must be an equality, that is, $V_\subl$ and $V_{\subg}$ have maximal dimension.
\end{itemize}

Now let $V_{\subl}$ and $V_{\subg}$ be any pair of transversal MTIS'es and $g_{\subl}$ any automorphism of $V_{\subl}$ with $\rho(g_{\subl}) < 1$. Let us show that there is at most one pseudohyperbolic map with frame $(V_{\subl}, V_{\subg})$ and dynamical part $g_{\subl}$. Indeed, let $g$ be such a map. Then we may calculate $V_{\subl}(g)$, $V_{\sube}(g)$, $V_{\subg}(g)$ and the restrictions of $g$ onto these subspaces, which determines $g$ uniquely. Indeed:

\begin{itemize}
\item By definition, $V_{\subl}(g) = V_{\subl}$ and $V_{\subg}(g) = V_{\subg}$.
\item Let $x_{\subl} \in V_{\subl}$, $x_{\sube} \in V_{\sube}(g)$. Then we have (denoting by $\langle \bullet, \bullet \rangle$ the bilinear form corresponding to the quadratic form $Q$):
\[\langle x_{\subl}, x_{\sube} \rangle
= \lim_{n \to +\infty} \langle g^n(x_{\subl}), g^n(x_{\sube}) \rangle
= \left\langle \lim_{n \to +\infty} g_{\subl}^n(x_{\subl}),\; x_{\sube} \right\rangle
= 0.\]
This shows that $V_{\sube}(g) \perp V_{\subl}$. In the same way, we get $V_{\sube}(g) \perp V_{\subg}$; hence $V_{\sube}(g) \subset V_{\subl}^\perp \cap V_{\subg}^\perp$. But clearly, the right-hand side is a space of dimension at most 1; hence $V_{\sube}(g) = V_{\subl}^\perp \cap V_{\subg}^\perp$.
\item By definition, $\restr{g}{V_{\subl}} = g_{\subl}$ and $\restr{g}{V_{\sube}}$ is the identity.
\item For $x \in \mathbb{R}^{d+1,d}$, we define $x_{\subl}, x_{\sube}, x_{\subg}$ to be the components of $x$ lying in $V_{\subl}(g)$, $V_{\sube}(g)$, $V_{\subg}(g)$ (so that $x = x_{\subl} + x_{\sube} + x_{\subg}$). For every $x$, since $Q(x_{\subl}) = Q(x_{\subg}) = \langle x_{\subl}, x_{\sube} \rangle = \langle x_{\subg}, x_{\sube} \rangle = 0$, we have
\[Q(x) = 2\langle x_{\subl}, x_{\subg} \rangle + Q(x_{\sube}).\]
Now if we apply $g$, we get:
\[Q(g(x)) = 2\langle g_{\subl}(x_{\subl}), g_{\subg}(x_{\subg}) \rangle + Q(x_{\sube}),\]
hence for every $x_{\subl} \in V_{\subl}$ and $x_{\subg} \in V_{\subg}$, we have $\langle x_{\subl}, x_{\subg} \rangle = \langle g_{\subl}(x_{\subl}), g_{\subg}(x_{\subg}) \rangle$. It follows that $g_{\subg}$ is adjoint to $g_{\subl}^{-1}$. More rigorously, we have
\begin{equation}
\label{eq:duality}
g_{\subg} = \Phi_\mathcal{V}^{-1} \circ (g_{\subl}^{-1})^* \circ \Phi_\mathcal{V},
\end{equation}
where $\Phi_\mathcal{V}: V_{\subg} \to V_{\subl}^*$ is the appropriate restriction and factoring of the canonical isomorphism $\Phi_Q: \mathbb{R}^{d+1,d} \to (\mathbb{R}^{d+1,d})^*$ defined by $\Phi_Q(x) \cdot y = \langle x, y \rangle$. This determines $g_{\subg}$ uniquely.
\end{itemize}

Finally, let $V_{\sube} := V_{\subl}^\perp \cap V_{\subg}^\perp$. Then $\dim V_{\sube} = 1$ and $\mathbb{R}^{d+1,d} = V_{\subl} \oplus V_{\sube} \oplus V_{\subg}$. Consider the map $g := g_{\subl} \oplus \Id_{V_{\sube}} \oplus g_{\subg}$, with $g_{\subg}$ defined by \eqref{eq:duality}. Then it is straightforward to check that $g$ is a pseudohyperbolic map with frame $(V_{\subl}, V_{\subg})$ and dynamical part $g_{\subl}$. (Note that it follows from \eqref{eq:duality} that the eigenvalues of $g_{\subg}$ are reciprocal to the eigenvalues of $g_{\subl}$).
\qed
\end{proof}

Incidentally, we can now prove --- as announced earlier --- that all pseudohyperbolic maps $g$ lie in $SO(d+1, d)$. Indeed, for all such $g$, we have $\det g_{\subg} = (\det g_{\subl})^{-1}$, hence $\det g = (\det g_{\subl})(\det \Id)(\det g_{\subg}) = 1$.

\begin{definition}
We define a \emph{frame} in general to be an ordered pair of transversal MTIS'es. If $\mathcal{V}$ is a frame, we write:
\begin{itemize}
\item $V_{\subl}$ its first component and $V_{\subg}$ its second component;
\item $V_{\sube}$ the line $V_{\subl}^\perp \cap V_{\subg}^\perp$;
\item $V_{\suble} := V_{\subl}^\perp = V_{\subl} \oplus V_{\sube}$ and $V_{\subge} := V_{\subg}^\perp = V_{\subg} \oplus V_{\sube}$.
\end{itemize}
\end{definition}

\subsection{Orientation}
\label{sec:orientation}

\begin{proposition}
\label{orientation}
It is possible to choose an orientation on all the MTIS'es $V$ (resp. on their orthogonal subspaces $V^\perp$), such that every $f \in SO^+(d+1,d)$ induces a direct isomorphism from $V$ to $f(V)$ (resp. from $V^\perp$ to $f(V^\perp) = f(V)^\perp$).
\end{proposition}
\begin{proof}
We first treat the case of the spaces orthogonal to the MTIS'es. We fix some orientations on $S$ and $T$ (recall that these are two mutually orthogonal maximal definite spaces, one positive and one negative). Then, for any MTIS $V$, $\pi_S$ induces an isomorphism from $V^\perp$ to $S$. Indeed, both spaces have dimension $d+1$, and
\[
\ker \restr{\pi_S}{V^\perp} = V^\perp \cap \ker \pi_S = V^\perp \cap T = \{0\},
\]
since $V^\perp$ is a positive and $T$ a negative definite subspace. We then choose the orientation of $V^\perp$ that makes $\restr{\pi_S}{V^\perp}$ a direct isomorphism.

Now consider the map from $S$ to itself given by the composition of
\[
S \overset{\pi_S^{-1}}{\longlongrightarrow} V^\perp
  \overset{f}{\longlongrightarrow}          f(V^\perp)
  \overset{\pi_S}{\longlongrightarrow}      S.
\]
It is easy to see that its determinant depends continuously on $f$ and never vanishes for $f \in SO^+(d+1,d)$. Since $SO^+(d+1,d)$ is connected, the determinant must have constant sign, hence the result.

Replacing $S$ by $T$, the same argument adapts for the MTIS'es themselves.
\qed
\end{proof}

From now on, let us fix such a family of orientations.

\begin{definition}
The \emph{positive wing} supported by a MTIS $V$ is the half-space
\[V^{\halfperp} := \setsuch{v + xe}{v \in V, \; x \geq 0},\]
where $e \in V^\perp$ is any vector such that whenever $(e_1, \ldots, e_d)$ is a direct basis of $V$, $(e_1, \ldots, e_d, e)$ is a direct basis of $V^\perp$. (The symbol $\halfperp$ should be read as "half-perp"; it is intended to represent half the symbol $\perp$.)
\end{definition}
\begin{proposition}
\label{wings_disjoint}
If $d$ is odd, the positive wings supported by any two transversal MTIS'es $V_{\subl}$ and $V_{\subg}$ have a trivial intersection:
\[V_{\subl}^{\halfperp} \cap V_{\subg}^{\halfperp} = \{0\}.\]
\end{proposition}

\begin{proof}
Let $\mathcal{B}_{\subg} = (e_{\subg}^1, \ldots, e_{\subg}^d)$ be any direct basis of $V_{\subg}$. We set
\[\mathcal{B}_{\subl} = (e_{\subl}^1, \ldots, e_{\subl}^d) := \Phi_\mathcal{V}^{-1}(-\mathcal{B}_{\subg}^*)\]
to be the basis of $V_{\subg}$ dual to the basis $-\mathcal{B}_{\subg} = (-e_{\subg}^1, \ldots, -e_{\subg}^d)$ (see the proof of Proposition \ref{frame_and_dynamical_part} for the definition of $\Phi_\mathcal{V}$). Let also $e_{\sube}$ be the vector of unit norm lying in $V_{\sube}$ such that $(\mathcal{B}_{\subg}, e_{\sube})$ is a direct basis of $V_{\subg}^\perp = V_{\subg} \oplus V_{\sube}$. We now define a basis of $\mathbb{R}^{d+1,d}$ by joining together these bases of $V_{\subl}$, $V_{\subg}$ and $V_{\sube}$:
\[\mathcal{B} := (e_{\subl}^1, \ldots, e_{\subl}^d, e_{\subg}^1, \ldots, e_{\subg}^d, e_{\sube}).\]
In this basis, the quadratic form $Q$ is then given by the matrix
\[\begin{pmatrix}
0    & -I_d & 0 \\
-I_d & 0    & 0 \\
0    & 0    & 1
\end{pmatrix}.\]
Now consider the automorphism $f$ given, in basis $\mathcal{B}$, by the matrix
\[\begin{pmatrix}
0   & I_d & 0      \\
I_d & 0   & 0      \\
0   & 0   & (-1)^d
\end{pmatrix}.\]
It is easy to show that $f \in SO^+(d+1,d)$ (for details, see \cite{AMS02}, proof of Lemma 3.1 --- they call this map $h_\pi$). But $f$ maps $\mathcal{B}_{\subg}$ onto $\mathcal{B}_{\subl}$ and $(\mathcal{B}_{\subg}, e_{\sube})$ onto $(\mathcal{B}_{\subl}, (-1)^de_{\sube})$. Hence by Proposition \ref{orientation}, the latter are direct bases of $V_{\subg}$ and $V_{\subg}^\perp$. This implies that
\[\begin{cases}
V_{\subl}^{\halfperp} = V_{\subl} + (-1)^d\mathbb{R}^{\geq 0}e_{\sube} \\
V_{\subg}^{\halfperp} = V_{\subg} + \mathbb{R}^{\geq 0}e_{\sube}.
\end{cases}\]
Hence
\[V_{\subl}^{\halfperp} \cap V_{\subg}^{\halfperp} = \left((-1)^d\mathbb{R}^{\geq 0} \cap \mathbb{R}^{\geq 0}\right)e_{\sube};\]
since $d$ is odd, the conclusion follows.
\qed
\end{proof}

\begin{remark}
Suppose now that $d$ is even. Then the same argument shows that two positive wings \emph{always} have a nontrivial intersection. Thus with our methods, there is no hope to construct a non-abelian free properly discontinuous subgroup in $SO(d+1, d) \rtimes \mathbb{R}^{2d+1}$ for even $d$, since a crucial point is the existence of $2n$ pairwise disjoint wings (for $n > 1$). Indeed, it was shown in \cite{AMS02} (Theorem A), using a very similar orientation argument, that such subgroups do not exist.
\end{remark}

\begin{definition}
\label{direction}
For every frame $\mathcal{V}$ (or $\mathcal{V}'$, $\mathcal{V}_i$, and so on), we denote by $e_{\sube}$ (resp. $e'_{\sube}$, $e_{i,\sube}$, and so on) the vector of unit norm contained in the half-line $V_{\sube} \cap V_{\subg}^{\halfperp}$. If $d$ is odd, we then have:\[\begin{cases}
V_{\subl}^{\halfperp} = V_{\subl} - \mathbb{R}^{\geq 0}e_{\sube} \\
V_{\subg}^{\halfperp} = V_{\subg} + \mathbb{R}^{\geq 0}e_{\sube}.
\end{cases}\]
\end{definition}

\subsection{Strength of contraction and other metric considerations}
\label{sec:metric}

From now on, we assume $d$ to be odd. We introduce on $\mathbb{R}^{d+1,d}$, in addition to its structural quadratic form $Q$, several positive definite quadratic forms. Every such form $N$ gives us an inner product (written $\langle x, y \rangle_N$), a Euclidean norm (written $\|x\|_N := N(x)^{\frac{1}{2}}$; hence also a metric on $\mathbb{R}^{d+1,d}$), and an operator norm (written also $\|g\|_N := \sup \frac{\|g(x)\|_N}{\|x\|_N}$).

First, we need a "global" norm, that we shall use most of the time: it will enable us to take measurements that do not depend on a particular frame. Insofar as all norms on a finite-dimensional space are equivalent, its choice does not really matter; however, the following particular expression will simplify some of the proofs. We define the form $N_0$ by
\begin{equation}
\label{eq:norm_definition}
\forall x \in \mathbb{R}^{d+1,d},\quad N_0(x) = N_S(\pi_S(x)) + N_T(\pi_T(x))
\end{equation}
(compare this with \eqref{eq:form_decomposition}).

However, for every frame $\mathcal{V}$, we also need a "local" norm, that will make calculations involving this frame easier. We define $N_\mathcal{V}$ to be the (positive definite) quadratic form on $\mathbb{R}^{d+1,d}$ that makes the spaces $V_{\subl}$, $V_{\sube}$ and $V_{\subg}$ pairwise orthogonal, but whose restriction to any of these spaces coincides with $N_0$.

Consider a vector space $E$ (for the moment, the reader may suppose that $E = \mathbb{R}^{d+1,d}$; later we will also need the case $E = \Lambda^d \mathbb{R}^{d+1,d}$). We define
\[\begin{cases}
\pi_\mathbb{S}: E \setminus \{0\} \to \mathbb{S}(E) \\
\pi_\mathbb{P}: E \setminus \{0\} \to \mathbb{P}(E)
\end{cases}\]
to be, respectively, the canonical projections onto the sphere $\mathbb{S}(E) := (E \setminus \{0\})/\mathbb{R}^{> 0}$ and the projective space $\mathbb{P}(E) := (E \setminus \{0\})/\mathbb{R}^*$. (Readers who think of the sphere as a subset of $E$ might get confused when we change the norm; this is why we define $\mathbb{S}(E)$ as an abstract quotient space.) For every linear map $g: E \to E$, we define the corresponding maps $g_\mathbb{S}: \mathbb{S}(E) \to \mathbb{S}(E)$ and $g_\mathbb{P}: \mathbb{P}(E) \to \mathbb{P}(E)$ (written simply $g$ when no confusion is possible.)

Consider a metric space $(X, \delta)$; let $A$ and $B$ be two subsets of $X$. We shall denote the ordinary, minimum distance between $A$ and $B$ by
\[\delta(A, B) := \inf_{a \in A} \inf_{b \in B} \delta(a, b),\]
as opposed to the Hausdorff distance, which we shall denote by
\[\delta^\mathrm{Haus}(A, B) := \max\left( \sup_{a \in A} \delta(a, B),\; \sup_{b \in B} \delta(b, A) \right).\]

For every positive definite quadratic form $N$ on $E$, for every $\overline{x}, \overline{y} \in \mathbb{S}(E)$, we define the distance
\[\alpha_N (\overline{x}, \overline{y}) := \arccos \frac{\langle x, y \rangle_N}{\|x\|_N \|y\|_N},\]
where $x$ and $y$ are any vectors representing respectively $\overline{x}$ and $\overline{y}$ (obviously, the value does not depend on the choice of $x$ and $y$). This measures the angle between the half-lines $\overline{x}$ and $\overline{y}$. For shortness' sake, we will usually simply write $\alpha_N(x, y)$ with $x, y \in E \setminus \{0\}$, to mean $\alpha_N (\pi_\mathbb{S}(x), \pi_\mathbb{S}(y))$.

In a similar way, we equip $\mathbb{P}(E)$ with the distance
\[\alpha_N^\mathrm{Proj} (x, y) := \alpha_N (\mathbb{R}x, \mathbb{R}y) = \min(\alpha_N(x, y),\; \alpha_N(x, -y)).\]
Note that for sets $X$ and $Y$ symmetric about the origin (such as vector spaces), we have $\alpha_N^\mathrm{Proj}(X, Y) = \alpha_N(X, Y)$: in this situation, we may ignore the distinction between the spherical and projective cases.

For any set $X \subset \mathbb{S}(E)$ and any radius $\eps > 0$, we shall denote the $\eps$-neighborhood of $X$ with respect to the distance $\alpha_N$ by:
\[B_N(X, \eps) := \setsuch{x \in \mathbb{S}(E)}{\alpha_N(x,X) < \eps}.\]
When $X$ is symmetric, we shall sometimes treat $B_N(X, \eps)$ as a subset of $\mathbb{P}(E)$.

For the sake of briefness, we shall often specify a "default" form at the beginning of some sections or paragraphs. In the rest of that section or paragraph, \emph{every} mention of \emph{any} of the metric-dependent values or functions defined above without explicit mention of the metric itself (such as $\langle x, y \rangle$, $\alpha(x, y)$, $B(X, \eps)$ and so on) is understood to refer to the current "default" metric.

Finally, we introduce the following notation. Let $A$ and $B$ be two positive quantities, and $p_1, \ldots, p_k$ some parameters. Whenever we write
\[A \ll_{p_1, \ldots, p_k} B,\]
we mean that there is a constant $C$, depending on nothing but $p_1, \ldots, p_k$, such that $A \leq CB$. (If we do not write any parameters, this means of course that $C$ is an absolute constant.) Whenever we write
\[A \asymp_{p_1, \ldots, p_k} B,\]
we mean that $A \ll_{p_1, \ldots, p_k} B$ and $B \ll_{p_1, \ldots, p_k} A$ at the same time.

\begin{definition}
Let $g$ be a pseudohyperbolic map, $\mathcal{V}$ its frame, $g_{\subl} = \restr{g}{V_{\subl}(g)}$ its dynamical part, $g_{\subg} = \restr{g}{V_{\subg}(g)}$. Since $V_{\subl}$ and $V_{\subg}$ are transversal, by Proposition \ref{wings_disjoint}, $V_{\subl}^{\halfperp}$ and $V_{\subg}^{\halfperp}$ have zero intersection. Their projections onto the sphere are then disjoint; being compact, they are always separated by a positive distance. We define the \emph{separation} of $\mathcal{V}$ (or, by abuse of terminology, of $g$) to be
\[\eps(g) = \eps(\mathcal{V}) := \alpha_{N_0}(V_{\subl}^{\halfperp}, V_{\subg}^{\halfperp}),\]
the distance between these projections in global metric. (The distance in local metric, $\alpha_{N_\mathcal{V}}(V_{\subl}^{\halfperp}, V_{\subg}^{\halfperp})$, is by definition always equal to $\frac{\pi}{2}$). For any constant $\eps > 0$, we say that $\mathcal{V}$ (or $g$) is \emph{$\eps$-separated} if $\eps(\mathcal{V}) \geq \eps$.

The \emph{strength of contraction} of $g$ is the quantity
\[s(g) := \max\left( \|g_{\subl}\|, \|g_{\subg}^{-1}\| \right)\]
(with the metric given indifferently by $N_0$ or $N_{\mathcal{V}(g)}$: both coincide on $V_{\subl}(g)$ and $V_{\subg}(g)$.) For $s > 0$, we say that $g$ is \emph{$s$-contracting} if $s(g) \leq s$. In this case, for all $x_{\subl} \in V_{\subl}(g)$ and $x_{\subg} \in V_{\subg}(g)$, we have
\[\frac{\|g(x_{\subl})\|}{\|x_{\subl}\|} \leq s \text{ and } \frac{\|g(x_{\subg})\|}{\|x_{\subg}\|} \geq s^{-1}.\]
\end{definition}

Note that if $d > 1$, there is no constant $C$ such that all pseudohyperbolic maps would be $C$-contracting, as the norm may be much larger than the spectral radius. However, for any pseudohyperbolic map $g$, we have
\[s(g^n) = \underset{n \to \infty}{O} \left( \rho(g_{\subl})^n \right) \underset{n \to \infty}{\to} 0.\]

Now we need to formulate an essential property of the metrics defined above, that we shall very often use subsequently. All of the norms $\| \bullet \|_{N_\mathcal{V}}$ and the associated distances $\alpha_{N_\mathcal{V}}$ are Lipschitz-equivalent, with a common Lipschitz constant that depends only on the separation of $\mathcal{V}$. More precisely:
\begin{lemma}
\label{uniformly_equivalent}
For every $\eps > 0$ and every $\eps$-separated frame $\mathcal{V}$, we have:
\[\forall x \in \mathbb{R}^{d+1,d},\quad
\|x\|_{N_\mathcal{V}} \asymp_{\eps} \|x\|_{N_0};\]
\[\forall x, y \in \mathbb{S}(\mathbb{R}^{d+1,d}),\quad
\alpha_{N_\mathcal{V}}(x, y) \asymp_{\eps} \alpha_{N_0}(x, y).\]
\end{lemma}
\begin{proof}
For any frame $\mathcal{V}$, let $C(\mathcal{V})$ be the Lipschitz constant between the norms given by $N_0$ and $N_\mathcal{V}$, \ie the smallest constant satisfying the first inequality above. Then $C(\mathcal{V})$ is always finite, and may be expressed as the operator norm of the identity map subordinated to the norms given by $N_0$ and $N_\mathcal{V}$: hence it depends continuously on $\mathcal{V}$. Since for any fixed $\eps > 0$, the set of all $\eps$-separated frames is compact, the first claim follows.

Now if two norms given by $N$ and $N'$ are $C$-Lipschitz-equivalent, then the corresponding distances $\alpha_N$ and $\alpha_{N'}$ are always $C^2$-Lipschitz-equivalent. Indeed, in dimension 2, this follows from a straightforward calculation; in the general case, we may simply fix two vectors $x$ and $y$ and restrict our attention to the subspace they span. Hence the second estimation follows from the first.
\qed
\end{proof}

\section{Pseudohyperbolicity of products}
\label{sec:group}

The goal of this section is to prove Proposition \ref{product_pseudohyperbolic}, which essentially states that under some conditions, the product of several pseudohyperbolic maps is still pseudohyperbolic.

\subsection{Proximal case}
\label{sec:proximal}

Let $E$ be a vector space. We fix a default quadratic form $\hat{N}_0$ on $E$. (In practice, we will apply the results of this subsection to $E = \Lambda^d \mathbb{R}^{d+1, d}$.)

Our first goal is to show Lemma \ref{product_proximal}, which is analogous to Proposition \ref{product_pseudohyperbolic} (and will be used to prove it), but with proximal maps instead of pseudohyperbolic ones. We begin by a few definitions.

\begin{definition}
Let $f \in GL(E)$. Let $\lambda$ be an eigenvalue of $f$ with maximal modulus. We say that $f$ is \emph{proximal} if $\lambda$ is unique and has multiplicity~1. We may then decompose $E$ into a direct sum of a line $V_s(f)$, called its \emph{attracting space}, and a hyperplane $V_u(f)$, called its \emph{repulsing space}, both stable by $f$ and such that:
\[\begin{cases}
\restr{f}{V_s} = \pm \lambda \Id \\
\text{for every eigenvalue } \mu \text{ of } \restr{f}{V_u},\; |\mu| < |\lambda|.
\end{cases}\]
We define the \emph{separation} of $f$ to be $\eta(f) := \alpha (V_s(f), V_u(f))$. For any constant $\eta > 0$, we say that $f$ is \emph{$\eta$-separated} if $\eta(f) \geq \eta$. For any quadratic form $N$ on $E$, we define the \emph{strength of contraction} of $f$ with respect to $N$ by
\[\hat{s}_N(f) := \frac{\|\restr{f}{V_u}\|_N}{|\lambda|}\]
(we remind that writing simply $\hat{s}$ means $\hat{s}_{\hat{N}_0}$.) Note that these definitions are different from the ones we used in the context of pseudohyperbolic maps (hence the new notation $\hat{s}$).
\end{definition}
\begin{definition}
An \emph{independent proximal system} is a tuple $F = (f_1, \ldots, f_n)$ of maps $f_i \in GL(E)$ such that:
\begin{enumerate}[(i)]
\item every $f_i$ and every $f_i^{-1}$ is proximal;
\item for every indices $i$, $i'$ and signs $\sigma$, $\sigma'$ such that $(i', \sigma') \neq (i, -\sigma)$, we have
\[\alpha(V_s(f_i^\sigma), V_u(f_{i'}^{\sigma'})) > 0.\]
\end{enumerate}
In this case, we define the \emph{separation} of $F$ to be
\[\eta(F) := \min_{(i', \sigma') \neq (i, -\sigma)} \alpha(V_s(f_i^\sigma), V_u(f_{i'}^{\sigma'})),\]
and the \emph{contraction strength} of $F$ to be
\[\hat{s}(F) := \max_{i, \sigma} \hat{s}(f_i^\sigma).\]
\end{definition}
\begin{definition}
\label{cyclically_reduced}
Take a nonnegative integer $k$, and take $k$ couples $(i_1, \sigma_1), \ldots, (i_k, \sigma_k)$ such that for every $l$, $1 \leq i_l \leq n$ and $\sigma_l = \pm 1$. Consider the word $f = f_{i_1}^{\sigma_1} \ldots f_{i_k}^{\sigma_k}$.

We say that $f$ is \emph{reduced} if for every $l$ such that $1 \leq l \leq k-1$, we have $(i_{l+1}, \sigma_{l+1}) \neq (i_l, -\sigma_l)$. We say that $f$ is \emph{cyclically reduced} if it is reduced and also satisfies $(i_1, \sigma_1) \neq (i_k, -\sigma_k)$.
\end{definition}

Now we prove an analog of Proposition \ref{product_pseudohyperbolic} in the proximal case:

\begin{lemma}
\label{product_proximal}
For every $\eta > 0$, there is a constant $\hat{s}(\eta) > 0$ with the following property. Let $F = (f_1, \ldots, f_n)$ be any $\eta$-separated, $\hat{s}(\eta)$-contracting independent proximal system. Let $f = f_{i_1}^{\sigma_1} \ldots f_{i_k}^{\sigma_k}$ (with $\sigma_l = \pm 1$) any nonempty cyclically reduced word.
Then $f$ is proximal, $\hat{s}(f) \ll_{\eta} \hat{s}(F)$ and
\[\alpha(V_s(f),\; V_s(f_{i_1}^{\sigma_1})) \;\ll_{\eta}\; \hat{s}(F).\]
\end{lemma}

Before proceeding, we need a technical lemma that relates the abstract strength of contraction $\hat{s}(f)$ and some actual Lipschitz constants of $f$ acting on the projective space $\mathbb{P}(E)$. For any set $X \subset \mathbb{P}(E)$, we introduce the following notation for the Lipschitz constant of $f$ restricted to $X$ in metric given by $N$:
\[\mathcal{L}_N(f, X) :=
\sup_{\substack{(x, y) \in X^2 \\ x \neq y}}
\frac{\alpha_N^\mathrm{proj}(f(x), f(y))}{\alpha_N^\mathrm{proj}(x, y)}\]

\begin{lemma}
\label{Lipschitz}
For any $\eta > 0$, $\zeta > 0$, for any proximal $\eta$-separated map $f$, we have :
\begin{subequations}
\label{eq:Lipschitz}
  \begin{equation}
  \label{eq:Lipschitz1}
    \mathcal{L} \left( f,\; \mathbb{P}(E) \setminus B (V_u(f), \zeta) \right) \ll_{\eta, \zeta} \hat{s}(f)
  \end{equation}
  \begin{equation}
  \label{eq:Lipschitz2}
    \hat{s}(f) \ll_{\eta, \zeta} \mathcal{L} \left( f,\; B (V_s(f), \zeta) \right)
  \end{equation}
\end{subequations}
(using of course the metric given by $\hat{N}_0$.)
\end{lemma}

\begin{proof}
Let $\eta > 0$, $\zeta > 0$. For every proximal $f$, we define on $E$ a quadratic form $\hat{N}_f$ that makes $V_s(f)$ and $V_u(f)$ orthogonal but coincides with $\hat{N}_0$ on these spaces. By an obvious generalization of Lemma \ref{uniformly_equivalent}, for every proximal $\eta$-separated map $f$, we have
\begin{equation}
\label{eq:N0_to_Nf}
\alpha^\mathrm{proj}_{\hat{N}_f} \asymp_{\eta} \alpha^\mathrm{proj}_{\hat{N}_0}.
\end{equation}
(Lemma \ref{uniformly_equivalent} referred to $\alpha$ rather than $\alpha^\mathrm{proj}$, but since both distances are locally equal, this makes little difference.) Now consider a proximal map $f$, and note the following facts:
\begin{itemize}
\item From \eqref{eq:N0_to_Nf}, it follows
\begin{align*}
\mathbb{P}(E) \setminus B (V_u(f),\; \zeta)\; &\subset\; \mathbb{P}(E) \setminus B_{\hat{N}_f} (V_s(f),\; \zeta'),\\
B (V_s(f),\; \zeta)\; &\supset\; B_{\hat{N}_f} (V_s(f),\; \zeta'),
\end{align*}
where $\zeta' = C\zeta$ for some constant $C$ depending only on $\eta$. Moreover, it is clear that $X \subset Y$ implies $\mathcal{L}(f, X) \leq \mathcal{L}(f, Y)$.
\item For all $X$, we have
\[\mathcal{L}_{\hat{N}_0}(f, X) \asymp_{\eta} \mathcal{L}_{\hat{N}_f}(f, X)\]
\item For any $\zeta' > 0$, we have
\[\mathcal{L}_{\hat{N}_f} \left( f,\; B_{\hat{N}_f} (V_s(f), \zeta') \right)
\asymp_{\zeta'} \hat{s}_{\hat{N}_f}(f).\]
Indeed, consider the projection $\pi_u: \mathbb{P}(E) \setminus V_u(f) \to V_u(f)$ parallel to $V_s(f)$, defined by $\pi_u(x_u:1) = x_u$ (with obvious notations). It induces a homeomorphism from $B_{\hat{N}_f} (V_s(f), \zeta')$ to the ball $\setsuch{x \in V_u(f)}{\|x\|_{\hat{N}_f} \leq \frac{1}{\tan \zeta'}}$. A straightforward calculation shows that the said homeomorphism is bilipschitz (with respect to the metrics $\alpha^\mathrm{proj}_{\hat{N}_f}$ and $\|\bullet\|_{\hat{N}_f}$), with a Lipschitz constant $C(\zeta')$ that does not at all depend on $f$ or $\eta$. On the other hand, the Lipschitz constant of the conjugate function $\pi_u \circ f \circ \pi_u^{-1}$ is nothing other than $\hat{s}_{\hat{N}_f}(f)$. Hence $f$ is Lipschitz-continuous with constant $C(\zeta')^2\hat{s}_{\hat{N}_f}(f)$, hence the conclusion.
\item Since $\hat{N}_f$ and $\hat{N}_0$ coincide on $V_u(f)$ and $V_s(f)$, we have $\hat{s}_{\hat{N}_f}(f) = \hat{s}_{\hat{N}_0}(f)$.
\end{itemize}
Now to show \eqref{eq:Lipschitz1}, we simply apply all these steps in succession, keeping in mind that
\[\mathbb{P}(E) \setminus B_{\hat{N}_f} (V_u(f), \zeta') = B_{\hat{N}_f} (V_s(f), \frac{\pi}{2} - \zeta').\]
To show \eqref{eq:Lipschitz2}, we apply the same steps in the reverse order.
\qed
\end{proof}

\begin{proof}[of Lemma \ref{product_proximal}]
Let $\eta > 0$, and let $F = (f_1, \ldots, f_n)$ be an $\eta$-separated, $\hat{s}(\eta)$-contracting independent proximal system (for a value $\hat{s}(\eta)$ to be specified later).

An immediate corollary of Lemma \ref{Lipschitz} is that for every $\eta$-separated proximal map $\phi$ and every $\zeta \leq \eta$, we have
\begin{equation}
\label{eq:eps6}
\phi \left( \mathbb{P}(E) \setminus B(V_u(\phi), \zeta) \right) \subset B\left( V_s(\phi),\; C\left( \eta, \zeta \right)\hat{s}(\phi) \right)
\end{equation}
for some constant $C(\eta, \zeta)$. Indeed, $V_s(\phi) \in \mathbb{P}(E) \setminus B(V_u(\phi), \zeta)$ is a fixed point of $\phi$ and $\mathrm{diam}(\mathbb{P}(E) \setminus B(V_u(\phi), \zeta)) \leq \frac{\pi}{2} \ll 1$.

Let $\eta' = C(\eta, \frac{\eta}{3})\hat{s}(F)$. For every $l$ in the range from $1$ to $k$, we set
\[\begin{cases}
X_l^- := B(V_u(f_{i_l}^{\sigma_l}), \frac{\eta}{3}) \\
X_l^+ := B(V_s(f_{i_l}^{\sigma_l}), \eta').
\end{cases}\]
Then by \eqref{eq:eps6}, for every $l$ we have $f_{i_l}^{\sigma_l}(\mathbb{P}(E) \setminus X_l^-) \subset X_l^+$. Since $\hat{s}(F) \leq \hat{s}(\eta)$, if we choose $\hat{s}(\eta)$ small enough, we may suppose that $\eta' \leq \frac{\eta}{3}$. Then for every $l$ we also have $X_l^+ \subset \mathbb{P}(E) \setminus X_{l-1}^-$ (since the word $f$ is reduced). By induction, it follows that
\[f(\mathbb{P}(E) \setminus X_k^-) \subset X_1^+.\]

Now by \eqref{eq:Lipschitz1}, we know that for every $l$
\begin{subequations}
\begin{equation}
\label{eq:lip_less_s}
\mathcal{L} \left( f_{i_l}^{\sigma_l},\; \mathbb{P}(E) \setminus X_l^- \right) \ll_{\eta} \hat{s}(F) \leq \hat{s}(\eta).
\end{equation}
Once again, choosing $\hat{s}(\eta)$ small enough, we may actually suppose that
\begin{equation}
\label{eq:lip_less_1}
\mathcal{L} \left( f_{i_l}^{\sigma_l},\; \mathbb{P}(E) \setminus X_l^- \right) < 1.
\end{equation}
\end{subequations}
Since $f$ is cyclically reduced, we have $X_1^+ \subset \mathbb{P}(E) \setminus X_k^-$; hence $X_1^+$ is stable by $f$ and, by induction, we get
\[\mathcal{L} \left( f,\; X_1^+ \right) < 1.\]

It follows that $f$ is proximal and $V_s(f) \in X_1^+$ (see \cite{Tits72}, Lemma 3.8 for a proof), which settles the first and third statement of the conclusion. On the other hand, it is easy to see that $V_u(f) \subset X_k^-$ (indeed, consider any point $x \in \mathbb{P}(E)$ belonging to $V_u(f)$ but not to $X_k^-$: then we would have $\lim_{n \to \infty} f^n(x) = V_s(f)$, which contradicts the fact that $V_u(f)$ is a stable subspace). But we know that
\begin{align*}
\alpha(X_1^+, X_k^-)
  &\geq \textstyle
    \alpha(V_s(f_{i_1}^{\sigma_1}), V_u(f_{i_k}^{\sigma_k})) - \eta' - \frac{\eta}{3}\\
  &\geq \textstyle
    \eta - \frac{\eta}{3} - \frac{\eta}{3}\\
  &= \textstyle \frac{\eta}{3},
\end{align*}
hence $f$ is $\frac{\eta}{3}$-separated.

This allows us to apply \eqref{eq:Lipschitz2} to $f$:
\[\hat{s}(f) \ll_{\eta} \mathcal{L} \left( f,\; B(V_s(f), \frac{\eta}{3}) \right).\]
We know that $B(V_s(f), \frac{\eta}{3}) \subset B(V_s(f_{i_1}^{\sigma_1}), \frac{2\eta}{3}) \subset \mathbb{P}(E) \setminus X_k^-$, hence
\[\mathcal{L} \left( f,\; B(V_s(f), \frac{\eta}{3}) \right)
   \leq \mathcal{L} \left( f,\; \mathbb{P}(E) \setminus X_k^- \right).\]
On the other hand, using \eqref{eq:lip_less_s} in combination with \eqref{eq:lip_less_1}, we get that
\[\mathcal{L} \left( f,\; \mathbb{P}(E) \setminus X_k^- \right)
  \ll_{\eta} \hat{s}(F).\]
Stringing together these inequalities, we get
\[\hat{s}(f) \ll_{\eta} \hat{s}(F),\]
which settles the second statement of the conclusion.
\qed
\end{proof}

\subsection{Pseudohyperbolic case}
\label{sec:pseudohyperbolic}

Throughout this section, we work by default in metric given by $N_0$.

\begin{definition}
We define a \emph{frameset} $\mathcal{W}$ to be a set of $n$ frames $\mathcal{V}_1, \ldots, \mathcal{V}_n$ whose $2n$ components $V_{1, \subl}, V_{1, \subg}, \ldots, V_{n, \subl}, V_{n, \subg}$ are pairwise transversal. We define the \emph{separation} $\eps(\mathcal{W})$ of the frameset to be the minimal separation between any two MTIS'es forming the frameset.

Let $\mathcal{W} = (\mathcal{V}_1, \ldots, \mathcal{V}_n)$ be a frameset. A \emph{group based on $\mathcal{W}$} is a group $G$ generated by pseudohyperbolic maps $g_1, \ldots, g_n$ with respective frames $\mathcal{V}_1, \ldots, \mathcal{V}_n$. For $s > 0$, we say that $G$ is \emph{$s$-contracting} if all of its generators are $s$-contracting; the \emph{contraction strength} of $G$ is the number
\[s(G) := \max_i s(g_i).\]
\end{definition}

\begin{remark} \mbox{ }
\begin{itemize}
\item By the "separation between $V$ and $V'$", we mean here the separation of the frame $(V, V')$. Take care that we take the minimum over all of the $\binom{2n}{2}$ possible pairings, not just the frames $\mathcal{V}_1, \ldots, \mathcal{V}_n$.
\item Lemma \ref{pairwise_transversal} guarantees that framesets with an arbitrarily large number of frames exist.
\end{itemize}
\end{remark}

\begin{proposition}
\label{product_pseudohyperbolic}
For every $\eps > 0$, there is a constant $s_1(\eps) > 0$ with the following property. Let $\mathcal{W}$ be any $\eps$-separated frameset, $G =\,<g_1, \ldots, g_n>$ any $s_1(\eps)$-contracting group based on $\mathcal{W}$, $g = g_{i_1}^{\sigma_1} \ldots g_{i_k}^{\sigma_k}$ (with $\sigma_l = \pm 1$) any nonempty cyclically reduced word.

Then $g$ is pseudohyperbolic, $\frac{\eps}{3}$-separated, 1-contracting, and
\[\alpha_{N_0}^\mathrm{Haus}(V_{\subg}(g), V_{\subg}(g_{i_1}^{\sigma_1}))\; \ll_{\eps}\; s(G).\]
\end{proposition}
\begin{definition}
Such a group will be called a \emph{pseudohyperbolic group}.
\end{definition}
\begin{remark}
A pseudohyperbolic group is always free. Indeed, take any reduced word formed on its generators. We may find a cyclically reduced word conjugate to it, and we then know that it is a pseudohyperbolic map. Hence it is not equal to the identity.
\end{remark}

The Proposition follows from Lemma \ref{product_proximal} applied to the space $E := \Lambda^d \mathbb{R}^{d+1, d}$. Indeed, there is a correspondence between pseudohyperbolic maps in $\mathbb{R}^{d+1, d}$ and proximal maps in $E$, as will be shown below.

For every map $g \in L(\mathbb{R}^{d+1, d})$, we define the corresponding map $\Lambda^d g \in L(E)$, and for every quadratic form $N$ on $\mathbb{R}^{d+1, d}$, we define the corresponding quadratic form $\Lambda^d N$ on $E$ by
\[\langle x_1 \wedge \ldots \wedge x_d,\; y_1 \wedge \ldots \wedge y_d \rangle_{\Lambda^d N}
:= \sum_{\sigma \in \mathcal{S}_d} \epsilon_\sigma \prod_{i=1}^d \langle x_i, y_{\sigma(i)} \rangle_N\]
(where $\mathcal{S}_d$ is the set of permutations of $\{1, \ldots, d\}$ and $\epsilon_\sigma$ stands for the signature of $\sigma$). We set the default form on $E$ to be $\hat{N}_0 = \Lambda^d N_0$. Let us now formulate the desired correspondence:

\begin{lemma} \mbox{ }
\label{pseudoh-to-prox}
\begin{enumerate}[(i)]
\item For $g \in SO(d+1,d)$, $\Lambda^d g$ is proximal iff $g$ is pseudohyperbolic. Moreover, the attracting (resp. repulsing) space of $\Lambda^d g$ depends on nothing but $V_{\subg}(g)$ (resp. $V_{\subl}(g)$):
\begin{equation}
\label{eq:frame_transformation}
\begin{cases}
V_s(\Lambda^d g) = \Lambda^d V_{\subg}(g) \\
V_u(\Lambda^d g) = \setsuch{x \in E}{x \wedge \Lambda^{d+1} V_{\suble}(g) = 0}.
\end{cases}
\end{equation}
\item For every $\eps > 0$, there is a constant $\eta(\eps) > 0$ such that for every $\eps$-separated frame $\mathcal{V}$, we have
\[\alpha (V_s, V_u) \geq \eta(\eps)\]
(with $V_s$ and $V_u$ defined as in \eqref{eq:frame_transformation}).
\item For every $\eps > 0$, for every $\eps$-separated pseudohyperbolic map $g \in SO(d+1,d)$, we have
\[s(g) \ll_{\eps} \hat{s}(\Lambda^d g).\]
If in addition $s(g) < 1$, we have
\[s(g) \asymp_{\eps} \hat{s}(\Lambda^d g).\]
\item For any two $d$-dimensional subspaces $V_1$ and $V_2$ of $\mathbb{R}^{d+1, d}$, we have
\[\alpha_{N_0}^\mathrm{Haus}(V_1, V_2)
\;\asymp\; \alpha_{\Lambda^d N_0}(\Lambda^d V_1, \Lambda^d V_2).\]
\end{enumerate}
\end{lemma}
\begin{proof} \mbox{ }
\begin{enumerate}[(i)]
\item Let $g \in SO(d+1,d)$. Let $\lambda_{1}, \ldots \lambda_{2d+1}$ be the eigenvalues of $g$ counted with multiplicity and ordered by increasing absolute value. Then we know that the eigenvalues of $\Lambda^d g$ counted with multiplicity are exactly the products of the form $\lambda_{i_1}\ldots\lambda_{i_d}$, where $1 \leq i_1 < \ldots < i_d \leq 2d+1$. As the two largest of them are $\lambda_{d+2} \ldots \lambda_{2d+1}$ and $\lambda_{d+1}\lambda_{d+3} \ldots \lambda_{2d+1}$, it follows that $\Lambda^d g$ is proximal iff $|\lambda_{d+1}| < |\lambda_{d+2}|$.

Suppose that this is the case. Being isotropic spaces, $V_{\subl}(g)$ and $V_{\subg}(g)$ have dimension at most $d$; it follows that $|\lambda_{d+1}| = 1$. We then have $|\lambda_{d+2}| > 1$, hence $\dim V_{\subg}(g) = d$. Since $V_{\sube}(g) \subset V_{\subg}(g)^\perp$ and $V_{\sube}(g)$ is transversal to $V_{\subg}(g)$, we get that $\dim V_{\sube}(g) = 1$. Having all this, it is easy to show that the identity \eqref{eq:duality} holds, hence $\lambda_{d+1} = \frac{\det g}{(\det g_{\subl})(\det g_{\subg})} = 1$. We conclude that $g$ is pseudohyperbolic. The converse is obvious.

As for the expression of $V_s$ and $V_u$, it follows immediately by considering a basis that trigonalises $g$.

\item Let $\eps > 0$. Clearly, $\alpha(V_s, V_u)$
depends continuously on $V_{\subl}$ and $V_{\subg}$, and never vanishes when $V_{\subl}$ and $V_{\subg}$ are transversal. Since the set of all $\eps$-separated frames is compact, this expression must have a positive lower bound.

\item Let $\eps > 0$; let $g \in SO(d+1, d)$ be an $\eps$-separated pseudohyperbolic map with frame $\mathcal{V}$. We proceed in three steps.

\begin{itemize}
\item First, note that, by \eqref{eq:duality}, we have $\|g_{\subg}^{-1}\| \asymp_{\eps} \|g_{\subl}\|$, hence
\begin{equation}
\label{eq:g_<_and_g_>}
s(g) \asymp_{\eps} \|g_{\subg}^{-1}\|.
\end{equation}

\item Second, let us show that for any proximal map $f$, we have
\begin{equation}
\label{eq:exterior_uniformly_equivalent}
\hat{s}_{\Lambda^d N_0}(f) \asymp_{\eps} \hat{s}_{\Lambda^d N_\mathcal{V}}(f).
\end{equation}
(Caution: $\Lambda^d N_\mathcal{V}$ is in general not the same as $\hat{N}_f$.) Indeed, note that if some norms given by $N$ and $N'$ are $C$-Lipschitz-equivalent, then the norms given by $\Lambda^d N$ and $\Lambda^d N'$ are $C^d$-Lipschitz-equivalent. The above inequalities then follow from Lemma \ref{uniformly_equivalent}.

\item The last step is to prove the result in metric given by $N_{\mathcal{V}(g)}$. Let $s_{1} \leq \ldots \leq s_{d}$ (resp. $s'_{1} \geq \ldots \geq s'_{d}$) be the singular values of $g_{\subg}$ (resp. $g_{\subl}$), so that $\|g_{\subl}\|_{N_\mathcal{V}} = \|g_{\subl}\|_{N_0} = s'_{1}$ and $\|g_{\subg}^{-1}\| = s_{1}^{-1}$. Since the spaces $V_{\subl}$, $V_{\sube}$ and $V_{\subg}$ are stable by $g$ and pairwise $N_\mathcal{V}$-orthogonal, we get that the singular values of $g$ in metric given by $N_\mathcal{V}$ are
\[s'_{d}, \ldots, s'_{1}, 1, s_{1}, \ldots, s_{d}\]
(note however that if we do not suppose $s(g) < 1$, this list might not be sorted in increasing order.) On the other hand, we know that the singular values of $\Lambda^d g$ in metric given by $\Lambda^d N_\mathcal{V}$ are products of $d$ distinct singular values of $g$ in metric given by $N_\mathcal{V}$. Since $V_s(\Lambda^d g)$ is $\Lambda^d N_\mathcal{V}$-orthogonal to $V_u(\Lambda^d g)$, we may once again analyze the singular values separately for each subspace. We know that the singular value corresponding to $V_s$ is equal to $s_{1} \ldots s_{d}$; we deduce that $\left\| \restr{\Lambda^d g}{V_u} \right\|_{\Lambda^d N_\mathcal{V}}$ is equal to the maximum of the remaining singular values. In particular it is larger than $1 \cdot s_{2} \ldots s_{d}$. On the other hand, if $\lambda$ is the largest eigenvalue of $\Lambda^d g$, then we have
\[|\lambda| = |\lambda_1 \ldots \lambda_d| = |\det g_{\subg}| = s_{1} \ldots s_{d}\]
(where $\lambda_1, \ldots, \lambda_d$ are the eigenvalues of $g_{\subg}$). It follows that:
\begin{equation}
\label{eq:s'_lower_bound}
\hat{s}_{\Lambda^d N_\mathcal{V}}(\Lambda^d g)
   = \frac{\left\| \restr{\Lambda^d g}{V_u} \right\|_{\Lambda^d N_\mathcal{V}}}{|\lambda|}
   \geq \frac{1 \cdot s_{2} \ldots s_{d}}{s_{1} \ldots s_{d}}
   = s_{1}^{-1}
   = \|g_{\subg}^{-1}\|.
\end{equation}
By combining \eqref{eq:g_<_and_g_>}, \eqref{eq:exterior_uniformly_equivalent} and \eqref{eq:s'_lower_bound}, we get the first estimation.

Now suppose that $s(g) < 1$. Then we have $s'_{1} \leq s(g) < 1$ and $1 < s(g)^{-1} \leq s_{1}$, which means that the singular values of $\Lambda^d g$ are indeed sorted in the "correct" order. Hence $1 \cdot s_{2} \ldots s_{d}$ is actually the largest singular value of $\restr{\Lambda^d g}{V_u}$, and the inequality becomes an equality: $\hat{s}_{\Lambda^d N_\mathcal{V}}(\Lambda^d g) = \|g_{\subg}^{-1}\|$. The second estimation follows.
\end{itemize}

\item Let $V_1$ and $V_2$ be two $d$-dimensional spaces. We introduce the notations:
\[\alpha_1 := \alpha_{N_0}^\mathrm{Haus}(V_1, V_2);\]
\[\alpha_2 := \alpha_{\Lambda^d N_0}(\Lambda^d V_1, \Lambda^d V_2).\]
We may find an $N_0$-orthonormal basis $(e_1, \ldots, e_{2d+1})$ of $\mathbb{R}^{d+1,d}$ such that $V_1$ has basis $(e_1, \ldots, e_d)$ and $V_2$ has basis
\[ \left( (\cos \theta_i) e_i + (\sin \theta_i) e_{d+i} \right)_{1 \leq i \leq d},\]
for some angles $0 \leq \theta_1 \leq \ldots \leq \theta_d \leq \frac{\pi}{2}$. In this case, we have:
\[\alpha_1 = \theta_d\]
and
\[\cos \alpha_2 = \prod_{i=1}^d \cos \theta_i,\]
hence
\[(\cos \alpha_1)^d \leq \cos \alpha_2 \leq \cos \alpha_1.\]
On the other hand, from the concavity of the function $y \mapsto (\arccos \exp y)^2$, it follows that for every $\theta \in [0, \frac{\pi}{2}]$, we have
\[ \arccos((\cos \theta)^d) \leq \sqrt{d}\theta.\]
Finally we get
\[\alpha_1 \leq \alpha_2 \leq \sqrt{d}\alpha_1,\]
QED.\qed
\end{enumerate}
\end{proof}

We may now prove the main Proposition.

\begin{proof}[of Proposition \ref{product_pseudohyperbolic}]
Let $\eps > 0$; let $\mathcal{W} = (\mathcal{V}_1, \ldots, \mathcal{V}_n)$ be an $\eps$-separated frameset and $G =\,<g_1, \ldots, g_n>$ be an $s_1(\eps)$-contracting group based on $\mathcal{W}$, for some constant $s_1(\eps)$ to be specified later. Let $g = g_{i_1}^{\sigma_1} \ldots g_{i_k}^{\sigma_k}$ be a nonempty cyclically reduced word.

For every $i$, take $f_i = \Lambda^d g_i$. Let us check that we may apply Lemma \ref{product_proximal}. Indeed:
\begin{itemize}
\item By Lemma \ref{pseudoh-to-prox} (i), $F = (f_1, \ldots, f_n)$ is an independent proximal system. (Conditions (i) and (ii) follow, respectively, from the first and second part of Lemma \ref{pseudoh-to-prox} (i).)
\item By Lemma \ref{pseudoh-to-prox} (ii), we have $\eta(F) \leq \eta(\eps)$; in other words, $F$ is $\eta(\eps)$-separated. We set $\eta = \eta(\eps)$: then "$\ll_\eta$" always implies "$\ll_\eps$".
\item Without loss of generality, we may suppose $s(G) < 1$. Then by Lemma \ref{pseudoh-to-prox} (iii), we have $\hat{s}(F) \ll_\eps s(G)$, which is in turn no greater than $s_1(\eps)$. If we choose $s_1(\eps)$ sufficiently small (since $\eta$ is entirely determined by $\eps$), we then have
\[\hat{s}(F) \leq \hat{s}(\eta).\]
\end{itemize}

Now let us deduce the conclusions of the Proposition \ref{product_pseudohyperbolic} from the conclusions of Lemma \ref{product_proximal}, applied to the word $\Lambda^d g = f_{i_1}^{\sigma_1} \ldots f_{i_k}^{\sigma_k}$:
\begin{itemize}
\item That $g$ is pseudohyperbolic follows from Lemma \ref{pseudoh-to-prox} (i).

\item Let us show that
\[\alpha^\mathrm{Haus}(V_{\subg}(g),\; V_{\subg}(g_{i_1}^{\sigma_1})) \ll_\eps s(G).\]
Indeed, we have:
\begin{align*}
\alpha^\mathrm{Haus}(V_{\subg}(g),\; V_{\subg}(g_{i_1}^{\sigma_1}))
  &\ll      \alpha(\Lambda^d V_{\subg}(g),\; \Lambda^d V_{\subg}(g_{i_1}^{\sigma_1})) &&\text{by Lemma \ref{pseudoh-to-prox} (iv)}\\
  &=        \alpha(V_s(\Lambda^d g),\; V_s(f_{i_1}^{\sigma_1}))           &&\text{by Lemma \ref{pseudoh-to-prox} (i)}\\
  &\ll_\eta \hat{s}(F)                                                    &&\text{by Lemma \ref{product_proximal}}\\
  &\ll_\eps s(G)                                                          &&\text{by Lemma \ref{pseudoh-to-prox} (iii);}\\
\end{align*}
and we know that "$\ll_\eta$" implies "$\ll_\eps$".

\item Let us show that $g$ is $\frac{\eps}{3}$-separated. Since $s(G) \leq s_1(\eps)$, we may choose $s_1(\eps)$ sufficiently small to deduce, from the previous point, the following inequality:
\[\alpha^\mathrm{Haus}(V_{\subg}(g),\; V_{\subg}(g_{i_1}^{\sigma_1})) \leq \frac{\eps}{3}.\]
Replacing $g$ by $g^{-1}$, we get similarly
\[\alpha^\mathrm{Haus}(V_{\subl}(g),\; V_{\subl}(g_{i_k}^{\sigma_k})) \leq \frac{\eps}{3}.\]
Finally, since $g$ is cyclically reduced and $\mathcal{W}$ is $\eps$-separated, we know that
\[\alpha(V_{\subg}(g_{i_1}^{\sigma_1}),\; V_{\subl}(g_{i_k}^{\sigma_k})) \geq \eps.\]
From these three inequalities, it follows that
\begin{equation}
\alpha(V_{\subl}(g), V_{\subg}(g)) \geq \frac{\eps}{3}.
\end{equation}

\item Let us show that $g$ is 1-contracting. Using Lemma \ref{product_proximal} and Lemma \ref{pseudoh-to-prox} (iii), we get
\begin{align*}
s(g) &\ll_{\eps(g)} \hat{s}(\Lambda^d g) &&\text{by Lemma \ref{pseudoh-to-prox} (iii)}\\
     &\ll_\eta \hat{s}(F)                &&\text{by Lemma \ref{product_proximal}}\\
     &\ll_\eps s(G)                      &&\text{by Lemma \ref{pseudoh-to-prox} (iii) (since $s(G) < 1$.}\\
\end{align*}
Since $\eps(g) \geq \frac{\eps}{3}$ and $\eta = \eta(\eps)$, we get $s(g) \ll_\eps s(G) \leq s_1(\eps)$. If we take $s_1(\eps)$ sufficiently small, we deduce that
\[s(g) < 1.\]
\end{itemize}
\qed
\end{proof}

\section{The "tennis ball" and generalized Schottky groups}
\label{sec:tennis_ball}

Let $\eps > 0$, and let $\mathcal{V}$ be a frame.
\begin{definition}
\label{tennis_ball_domains}
We define, on the sphere $\mathbb{S}(\mathbb{R}^{d+1,d})$ (from now on simply referred to as $\mathbb{S}$), the following domains:
\[\begin{cases}
\mathcal{H}_\mathbb{S}^- := B_{N_\mathcal{V}}(\pi_\mathbb{S}(V_{\subl}^{\halfperp}), \eps) \\
\mathcal{H}_\mathbb{S}^+ := B_{N_\mathcal{V}}(\pi_\mathbb{S}(V_{\subg}^{\halfperp}), \eps).
\end{cases}\]
(Of course, they depend on $\mathcal{V}$ and $\eps$, but to simplify the notations, we shall leave this dependence implicit.) We call them \emph{tennis ball domains} (to understand why, draw them for $d=1$).
\end{definition}

In the following Proposition and its proof, we work in metric given by $N_{\mathcal{V}(g)}$.
\begin{proposition}
\label{tennis_ball}
For every $\eps > 0$, there is a constant $s_2(\eps)$ such that for any $s_2(\eps)$-contracting pseudohyperbolic map $g$ (with frame $\mathcal{V}$), we have
\begin{subequations}
\label{eq:tennis_ball}
  \begin{equation}
  \label{eq:tb1}
    g_\mathbb{S} \left( \mathbb{S} \setminus
       \overline{\mathcal{H}_\mathbb{S}^-} \right)
    \subset      \mathcal{H}_\mathbb{S}^+
  \end{equation}
  \begin{equation}
  \label{eq:tb2}
    g^{-1}_\mathbb{S} \left( \mathbb{S} \setminus
       \overline{\mathcal{H}_\mathbb{S}^+} \right)
    \subset      \mathcal{H}_\mathbb{S}^-.
  \end{equation}
\end{subequations}
\end{proposition}
\begin{remark}
Since we work here in metric given by $N_\mathcal{V}$, the separation of $g$ does not matter and $\eps$ has nothing to do with it. Instead $\eps$ defines the "aperture" of the tennis ball domains $\mathcal{H}_\mathbb{S}^\pm$.
\end{remark}
\begin{remark}
\label{tb12}
As $g_\mathbb{S}$ is a homeomorphism and the domains under consideration are regular, these two relations are actually equivalent. Also, since $V_{\subl}(g^{-1}) = V_{\subg}$, $V_{\subg}(g^{-1}) = V_{\subl}$ and $s(g^{-1}) = s(g)$, \eqref{eq:tb2} is nothing else than \eqref{eq:tb1} applied to $g^{-1}$.
\end{remark}

Let $\eps > 0$ and $g$ be a pseudohyperbolic map with frame $\mathcal{V}$. As previously done, for $x \in \mathbb{R}^{d+1,d}$, we define the triple $(x_{\subl}, x_{\sube}, x_{\subg}) \in V_{\subl} \times V_{\sube} \times V_{\subg}$ such that $x_{\subl} + x_{\sube} + x_{\subg} = x$; these are the $N_\mathcal{V}$-orthogonal projections of $x$ on the corresponding spaces. The vector $e_{\sube}$ (see Definition \ref{direction}) gives an orientation on $V_{\sube}$, which allows us to define an order on this 1-dimensional space: we say that $x_{\sube} \geq y_{\sube}$ iff $\langle x_{\sube}, e_{\sube} \rangle \geq \langle y_{\sube}, e_{\sube} \rangle$.

\begin{lemma}
\label{banana_description}
Let $x \in \mathbb{R}^{d+1,d} \setminus \{0\}$. Then we have:
\begin{subequations}
\label{eq:banana_description}
\begin{equation}
\label{eq:bd1}
\pi_\mathbb{S}(x) \in B (\pi_\mathbb{S}(V_{\subl}^{\halfperp}), \eps)
\iff \begin{cases}
     x_{\sube} \leq 0 \text{ \rm and } \frac{\|x_{\subg}\|}{\|x_{\subl} + x_{\sube}\|} < \tan \eps \\
     \text{\rm or} \\
     x_{\sube} \geq 0 \text{ \rm and } \frac{\|x_{\subg} + x_{\sube}\|}{\|x_{\subl}\|} < \tan \eps
     \end{cases}
\end{equation}
and
\begin{equation}
\label{eq:bd2}
\pi_\mathbb{S}(x) \in B (\pi_\mathbb{S}(V_{\subg}^{\halfperp}), \eps)
\iff \begin{cases}
     x_{\sube} \leq 0 \text{ \rm and } \frac{\|x_{\subl} + x_{\sube}\|}{\|x_{\subg}\|} < \tan \eps \\
     \text{\rm or} \\
     x_{\sube} \geq 0 \text{ \rm and } \frac{\|x_{\subl}\|}{\|x_{\subg} + x_{\sube}\|} < \tan \eps
     \end{cases}
\end{equation}
\end{subequations}
(and by replacing everywhere "$< \tan \eps$" by "$\leq \tan \eps$", we may characterize in a similar way the closures of these domains.)
\end{lemma}

\begin{proof}
Without loss of generality, let us concentrate on \eqref{eq:bd2} (the other statement follows simply by interchanging $V_{\subl}$ and $V_{\subg}$ and swapping the orientation of $V_{\sube}$.) Remember that $x \in V_{\subg}^{\halfperp}$ iff $x_{\subl} = 0$ and $x_{\sube} > 0$.
\begin{itemize}
\item Suppose $x_{\sube} \geq 0$. As $V_{\subg}^{\halfperp} \subset V_{\subge}$, we have $\alpha(x, V_{\subg}^{\halfperp}) \geq \alpha(x, V_{\subge})$. On the other hand, we have $\alpha(x, V_{\subge}) = \alpha(x, x_{\subg} + x_{\sube})$ and $x_{\subg} + x_{\sube} \in V_{\subg}^{\halfperp}$, which shows the opposite inequality. Hence $\alpha(x, V_{\subg}^{\halfperp}) = \alpha(x, V_{\subge})$.
\item Suppose $x_{\sube} \leq 0$; without loss of generality, we may assume that $\|x\| = 1$. Since $V_{\subg} \subset V_{\subg}^{\halfperp}$, obviously $\alpha(x, V_{\subg}^{\halfperp}) \leq \alpha(x, V_{\subg})$. Now let $y \in V_{\subg}^{\halfperp}$; then we have:
\begin{align*}
\cos \alpha(x, y) &=    \frac{\langle x, y \rangle}{\|y\|} \\
             &=    \frac{\langle x_{\subl} + x_{\subg} + x_{\sube},\; y_{\subg} + y_{\sube} \rangle}{\|y_{\subg} + y_{\sube}\|} \\
             &=    \frac{\langle x_{\subg}, y_{\subg} \rangle + \langle x_{\sube}, y_{\sube} \rangle}{\|y_{\subg} + y_{\sube}\|} \\
             &\leq \frac{\langle x_{\subg}, y_{\subg} \rangle}{\|y_{\subg}\|} \\
             &=    \cos \alpha(x, y_{\subg}),
\end{align*}
since $\langle x_{\sube}, y_{\sube} \rangle \leq 0$ and $\|y_{\subg} + y_{\sube}\| \geq \|y_{\subg}\|$. Hence $\alpha(x, y) \geq \alpha(x, y_{\subg})$, with $y_{\subg} \in V_{\subg}$. This shows the opposite inequality. Hence $\alpha(x, V_{\subg}^{\halfperp}) = \alpha(x, V_{\subg})$.
\end{itemize}
The result now follows from the fact that for any vector subspace $E \subset \mathbb{R}^{d+1,d}$, we have
\[\alpha(x, E) = \alpha(x, x_E) = \arccos \frac{\|x_E\|}{\|x\|} = \arctan \frac{\|x - x_E\|}{\|x_E\|},\]
where $x_E$ is the $N_\mathcal{V}$-orthogonal projection of $x$ onto $E$.
\qed
\end{proof}

\begin{proof}[of Proposition \ref{tennis_ball}]
By virtue of Remark \ref{tb12}, it is enough to show \eqref{eq:tb1}. Let $x \in \mathbb{R}^{d+1,d} \setminus \{0\}$ such that $\alpha(x, V_{\subl}^{\halfperp}) > \eps$; it is enough to prove that if $s(g) \leq s_2(\eps)$ (for a value of $s_2(\eps)$ to be specified later), we have $\alpha(g(x), V_{\subg}^{\halfperp}) < \eps$.

Suppose that $x_{\sube} \leq 0$. Then we have, by Lemma \ref{banana_description},
\[
\frac{\|x_{\subg}\|}{\|x_{\subl} + x_{\sube}\|} > \tan \eps.
\]
We deduce that
\begin{align*}
\frac{\|g(x)_{\subl} + g(x)_{\sube}\|^2}{\|g(x)_{\subg}\|^2} &= \frac{\|g(x_{\subl})\|^2 + \|g(x_{\sube})\|^2}{\|g(x_{\subg})\|^2} \\
                                           &\leq \frac{s(g)^2\|x_{\subl}\|^2 + \|x_{\sube}\|^2}{s(g)^{-2}\|x_{\subg}\|^2} \\
                                           &\leq s(g)^2\frac{\|x_{\subl}\|^2 + \|x_{\sube}\|^2}{\|x_{\subg}\|^2} \\
                                           &\leq s(g)^2(\tan \eps)^{-2} \\
                                           &<    (\tan \eps)^2,
\end{align*}
provided that $s(g) < (\tan \eps)^4$, which is true if we take $s_2(\eps)$ to be smaller than this value. Hence $\alpha(g(x), V_{\subg}) < \eps$. On the other hand, we have $g(x)_{\sube} = g(x_{\sube}) = x_{\sube} \leq 0$. It follows that $\alpha(g(x), V_{\subg}^{\halfperp}) < \eps$. In the case where $x_{\sube} \geq 0$, a completely analogous calculation yields the same result.
\qed
\end{proof}

Now consider a frameset $\mathcal{W} = (\mathcal{V}_1, \ldots, \mathcal{V}_n)$ and a set of radii $\eps_1, \ldots, \eps_n$.
\begin{definition}
\label{tennis_ball_sets}
Just as in Definition \ref{tennis_ball_domains}, we define for every index $i$ the domains
\[\begin{cases}
\mathcal{H}_{\mathbb{S}, i}^- := B_{N_{\mathcal{V}_i}}(\pi_\mathbb{S}(V_{i, \subl}^{\halfperp}), \eps_i) \\
\mathcal{H}_{\mathbb{S}, i}^+ := B_{N_{\mathcal{V}_i}}(\pi_\mathbb{S}(V_{i, \subg}^{\halfperp}), \eps_i).
\end{cases}\]
(Once again, they depend on $\mathcal{W}$ and the $\eps_i$, but to simplify the notations, we keep this dependence implicit.)
Let $G =\,<g_1, \ldots, g_n>$ be any group based on $\mathcal{W}$. If the sets $\overline{\mathcal{H}_{\mathbb{S}, i}^\pm}$ are pairwise disjoint and for every $i$, $s(g_i)$ is small enough to apply Proposition \ref{tennis_ball}, we say that $G$ is \emph{$(\eps_1, \ldots, \eps_n)$-Schottky}.
\end{definition}

In this case, it follows from Proposition \ref{tennis_ball} that $G$ is free. Indeed, we have for every $i$:
\begin{equation}
\label{eq:sph_ping_pong}
\begin{cases}
g_i      \left( \mathbb{S} \setminus \overline{\mathcal{H}_{\mathbb{S}, i}^-} \right)
  \subset \mathcal{H}_{\mathbb{S}, i}^+ \\
g_i^{-1} \left( \mathbb{S} \setminus \overline{\mathcal{H}_{\mathbb{S}, i}^+} \right)
  \subset \mathcal{H}_{\mathbb{S}, i}^-,
\end{cases}
\end{equation}
and we may apply the ping-pong lemma (see for example \cite{Tits72}, Proposition 1.1).

\section{Affine deformations}
\label{sec:affine}

\begin{definition}
Let $G \subset SO(d+1, d)$ be any linear group. An \emph{affine deformation} of $G$ is any group $\Gamma \subset \mathbb{R}^{d+1,d} \rtimes SO(d+1, d)$ such that the canonical projection $L: \mathbb{R}^{d+1,d} \rtimes SO(d+1, d) \to SO(d+1, d)$ induces an isomorphism from $\Gamma$ to $G$. In other terms, it is a group of affine transformations that does not contain pure translations and whose linear parts form the group $G$.
\end{definition}

Now suppose $G =\,<g_1, \ldots, g_n>$ is a free group; let $\Gamma$ be any affine deformation of $G$. Then it is generated by the elements $\gamma_1, \ldots, \gamma_n$ whose linear parts are $g_1, \ldots, g_n$, respectively. This means that, $G$ being fixed, $\Gamma$ is entirely determined by the translational parts of its generators, namely the vectors $\gamma_1(0), \ldots, \gamma_n(0)$. Reciprocally, for any family of vectors ${\bf t} = (t_1, \ldots, t_n) \in \left(\mathbb{R}^{d+1, d}\right)^n$, we may define $\gamma_1, \ldots, \gamma_n$ by $\gamma_i(x) = g_i(x) + t_i$ for all $i$. Since $G$ is free, the group generated by these elements is then an affine deformation of $G$, that we shall call $G({\bf t})$. This defines a bijection between the set of all affine deformations of $G$ and $\left(\mathbb{R}^{d+1, d}\right)^n$.

We may now state the Main Theorem more precisely:
\begin{theorem}
\label{main_theorem}
For every $\eps > 0$, there is a constant $s_3(\eps)$ with the following property. Let $\mathcal{W}$ be any $\eps$-separated frameset, $G$ any $s_3(\eps)$-contracting group based on $\mathcal{W}$. Then we can say that:
\begin{enumerate}[(i)]
\item The group $G$ is free;
\item There is a nonempty open set ${\bf T} \subset \left(\mathbb{R}^{d+1, d}\right)^n$ (depending on $G$) such that for every ${\bf t} \in {\bf T}$, the affine deformation $G({\bf t})$ acts properly discontinuously on $\mathbb{R}^{d+1, d}$;
\item For ${\bf t} \in {\bf T}$, the quotient space $\mathbb{R}^{d+1, d}/G({\bf t})$ is homeomorphic to a solid $(2d+1)$-dimensional handlebody with $n$ handles.
\end{enumerate}
\end{theorem}

\begin{proof}
We begin by giving a few definitions and notations, to be fixed for the remainder of this section.
\begin{itemize}
\item We fix $\eps > 0$, $\mathcal{W}$ an $\eps$-separated frameset, $G$ an $s_3(\eps)$-contracting group based on $\mathcal{W}$. We will determine the value of $s_3(\eps)$ in the course of the proof.
\item For every $i$, we choose a constant $\eps_i > 0$ such that for any set $X \subset \mathbb{S}$, we have
\begin{equation}
\label{eq:eps_i_and_eps}
B_{N_{\mathcal{V}_i}}(X, \eps_i) \subset B_{N_0}(X, \frac{\eps}{3}).
\end{equation}
We define the tennis-ball domains $\mathcal{H}_{\mathbb{S}, i}^\sigma$ accordingly (see definition \ref{tennis_ball_sets}).
\item For any ${\bf t} \in \left(\mathbb{R}^{d+1, d}\right)^n$, for any $i$ such that $1 \leq i \leq n$ and $\sigma = \pm 1$, we introduce the following domain. It is a subset of $\mathbb{R}^{d+1, d}$ constructed as a cone, whose apex depends on the translational part and whose base is the corresponding tennis-ball domain:
\[\mathcal{H}_i^\sigma({\bf t}) := \pi_\mathbb{S}^{-1}(\mathcal{H}_{\mathbb{S}, i}^\sigma) + \sigma u_i,\]
where $u_i$ is the solution to the equation $u_i + g_i(u_i) = t_i$. (Since $g_i$ is pseudohyperbolic, it does not have $-1$ as an eigenvalue, so that the equation has indeed a unique solution.)
\item For any such ${\bf t}$, we also introduce, for every index $i$, the domains 
\[\begin{cases}
\mathcal{\tilde{H}}_i^-({\bf t}) := \mathcal{H}_i^-({\bf t}) \\
\mathcal{\tilde{H}}_i^+({\bf t}) := \gamma_i \left(\mathbb{R}^{d+1,d} \setminus \overline{\mathcal{H}_i^-({\bf t})}\right)
\end{cases}\]
(where $\gamma_i: x \mapsto g_i(x) + t_i$ is the $i$-th generator of the affine deformation $G({\bf t})$: it depends implicitly on ${\bf t}$.) We also introduce the domain
\[\mathcal{H}^0 := \mathbb{R}^{d+1,d} \setminus \bigcup_{i=1}^n \bigcup_{\sigma = \pm} \overline{\mathcal{\tilde{H}}_i^\sigma}.\]
\item We define ${\bf T}$ to be the set of all ${\bf t} \in \left(\mathbb{R}^{d+1, d}\right)^n$ such that the $2n$ sets $\overline{\mathcal{H}_i^\pm({\bf t})}$ are pairwise disjoint.
\end{itemize}
Now (i) follows immediately from either Proposition \ref{product_pseudohyperbolic}, or Proposition \ref{tennis_ball} combined with \eqref{eq:eps_i_and_eps}. The claim (ii) follows from Lemma \ref{regions_disjoint_nonempty}, Lemma \ref{regions_disjoint_open} and Proposition \ref{fundamental_region} below (since the existence of a fundamental domain is equivalent to proper discontinuity). The latter Proposition is interesting in its own right, as it describes the exact shape of the fundamental domain. It also allows us to prove (iii). Indeed, if the fundamental domain is $\mathcal{H}^0$, then the quotient space $\mathbb{R}^{d+1, d}/G({\bf t})$ is homeomorphic to the space obtained from $\overline{\mathcal{H}^0}$ by identifying for every $i$ the border of $\mathcal{\tilde{H}}_i^-$ with the border of $\mathcal{\tilde{H}}_i^+$. But clearly, the borders of $\mathcal{\tilde{H}}_i^-$ and $\mathcal{\tilde{H}}_i^+$ are homeomorphic to $\mathbb{R}^{2d}$ (or, if you wish, to $2d$-dimensional open "disks"), and $\mathcal{H}^0$ is homeomorphic to $\mathbb{R}^{2d+1}$ (or a $2d+1$-dimensional open ball).
\qed
\end{proof}

\begin{lemma}
\label{regions_disjoint_nonempty}
The set ${\bf T}$ is nonempty.
\end{lemma}
\begin{proof}
Recall that $e_{i, \sube}$ is the vector of unit norm $N_0$ fixed by $g_i$ with a suitably chosen sign (see Definition \ref{direction}). We set
\[{\bf t}_0 := (2e_{1, \sube}, \ldots, 2e_{n, \sube}).\]
Then we have, for every $i$, $u_i = e_{i, \sube}$, since by definition $g_i(e_{i, \sube}) = e_{i, \sube}$.

Let us show that ${\bf t}_0 \in {\bf T}$, that is, that the sets $\overline{\mathcal{H}_i^\pm({\bf t}_0)}$ are pairwise disjoint. To do this, we include them in the sets $\pi_\mathbb{S}^{-1}(\mathcal{H}_{\mathbb{S}, i}^\pm)$, that we already know to be pairwise disjoint.

Indeed, let us fix an index $i$ and a sign $\sigma$. In this proof, we work in metric given by $N_{\mathcal{V}_i}$. We need to show that:
\[
\overline{\mathcal{H}_i^\sigma({\bf t}_0)}
\; := \; \overline{\pi_\mathbb{S}^{-1}(\mathcal{H}_{\mathbb{S}, i}^\sigma)} + \sigma e_{i, \sube}
\; \subset \; \pi_\mathbb{S}^{-1}(\mathcal{H}_{\mathbb{S}, i}^\sigma).
\]
Suppose, without loss of generality, that $\sigma = +1$; let $x \in \overline{\pi_\mathbb{S}^{-1}(\mathcal{H}_{\mathbb{S}, i}^+)}$. If $x = 0$, clearly, we have
\[x + e_{i, \sube} = e_{i, \sube} \in V_{i, \subg}^{\halfperp} \subset \pi_\mathbb{S}^{-1}(\mathcal{H}_{\mathbb{S}, i}^+).\]
Otherwise, it is easy to see that $\pi_\mathbb{S}(x) \in \overline{\mathcal{H}_{\mathbb{S}, i}^+}$. Clearly, we may apply Lemma \ref{banana_description}, provided we replace strict inequalities with non-strict ones. We reuse the notation $x = x_{\subl} + x_{\sube} + x_{\subg}$ from that lemma. Let us distinguish three cases:
\begin{itemize}
\item If $x_{\sube} \geq 0$, then we have
\[\frac{\|x_{\subl}\|}{\|x_{\subg} + x_{\sube}\|} \leq \tan \eps_i.\]
We still have $x_{\sube} + e_{\sube} \geq 0$ and $\|x_{\sube} + e_{\sube}\| > \|x_{\sube}\|$, hence
\[\frac{\|x_{\subl}\|}{\|x_{\subg} + x_{\sube} + e_{\sube}\|} < \frac{\|x_{\subl}\|}{\|x_{\subg} + x_{\sube}\|} \leq \tan \eps_i,\]
and we conclude that $x + e_{\sube} \in \pi_\mathbb{S}^{-1}(B_{N}(\pi_\mathbb{S}(V_{\subg}^{\halfperp}), \eps_i))$.
\item If $x_{\sube} \leq -e_{\sube}$, then, similarly,
\[\frac{\|x_{\subl} + x_{\sube} + e_{\sube}\|}{\|x_{\subg}\|} < \frac{\|x_{\subl} + x_{\sube}\|}{\|x_{\subg}\|} \leq \tan \eps_i,\]
and we reach the same conclusion.
\item If $-e_{\sube} < x_{\sube} < 0$, then we have
\[\frac{\|x_{\subl} + x_{\sube}\|}{\|x_{\subg}\|} \leq \tan \eps_i,\]
from which we deduce
\begin{align*}
\|x_{\subl}\| &< \|x_{\subl} + x_{\sube}\| \\
          &\leq (\tan \eps_i) \|x_{\subg}\| \\
          &< (\tan \eps_i) \|x_{\subg} + x_{\sube} + e_{\sube}\|;
\end{align*}
since $x_{\sube} + e_{\sube} \geq 0$, we reach again the same conclusion.
\qed
\end{itemize}
\end{proof}

\begin{lemma}
\label{regions_disjoint_open}
The set ${\bf T}$ is open.
\end{lemma}
\begin{proof}
Let ${\bf t}_0 = (t_{0,1}, \ldots, t_{0,n})$ be any element of ${\bf T}$. We know that any two of the sets $\overline{\mathcal{H}_i^\pm({\bf t}_0)}$ are disjoint; we claim that they are separated by a positive distance. Indeed, take any ball $B$ whose radius is large compared to ${\bf t}_0$. Then the parts that fall inside $B$ are compact and disjoint, hence separated by a positive distance. As for the parts that fall outside $B$, they are separated because asymptotically, their projections onto $\mathbb{S}$ --- namely $\overline{\mathcal{H}_{\mathbb{S}, i}^\pm}$ --- are also compact and disjoint.

Let $d_{\min}$ be the smallest of these distances. Consider the set of all ${\bf t} = (t_1, \ldots, t_n)$ such that for every index $i$, $\|u_i - u_{0,i}\| < \frac{d_{\min}}{2}$. Then clearly this set is a neighborhood of ${\bf t}_0$, and is included in ${\bf T}$.
\qed
\end{proof}

\begin{proposition}
\label{fundamental_region}
For any ${\bf t} \in {\bf T}$, the action of the affine deformation $\Gamma := G({\bf t})$ on the affine space $\mathbb{R}^{d+1,d}$ has fundamental domain $\mathcal{H}^0$. More precisely:
\begin{enumerate}[(i)]
\item The images of $\mathcal{H}^0$ under the elements of $\Gamma$ are pairwise disjoint;
\item The images of its closure cover the whole space:
\[\bigcup_{\gamma \in \Gamma} \overline{\gamma(\mathcal{H}^0)} = \mathbb{R}^{d+1,d}.\]
\end{enumerate}
\end{proposition}

\begin{proof}
Let us fix a value ${\bf t} = (t_1, \ldots, t_n) \in {\bf T}$; we call $\Gamma =\;<\gamma_1, \ldots, \gamma_n>\;:= G({\bf t})$ the corresponding affine deformation. The first thing to understand is that the domains $\mathcal{H}_i^\pm$ (from now on, we shall no longer mention the dependence on ${\bf t}$) satisfy "ping-pong identities" similar to \eqref{eq:sph_ping_pong}. Namely, it follows from \eqref{eq:sph_ping_pong} that
\begin{align*}
\gamma_i \left( \mathbb{R}^{d+1,d} \setminus \overline{\mathcal{H}_i^-} \right)
  &= \gamma_i \left( \mathbb{R}^{d+1,d} \setminus \overline{\pi_\mathbb{S}^{-1}(\mathcal{H}_{\mathbb{S}, i}^-) - u_i} \right) \\
  &= g_i \left( \pi_\mathbb{S}^{-1} \left( \mathbb{S} \setminus
            \overline{\mathcal{H}_{\mathbb{S}, i}^-} \right)\right) - g_i(u_i) + t_i \\
  &\subset \pi_\mathbb{S}^{-1}(\mathcal{H}_{\mathbb{S}, i}^+) + u_i \\
  &= \mathcal{H}_i^+,
\end{align*}
and the same holds for $\gamma_i^{-1}$. Thus we get:
\begin{equation}
\label{eq:aff_ping_pong}
\begin{cases}
\gamma_i     \left( \mathbb{R}^{d+1,d} \setminus \overline{\mathcal{H}_i^-} \right)
  \subset \mathcal{H}_i^+ \\
\gamma_i^{-1}\left( \mathbb{R}^{d+1,d} \setminus \overline{\mathcal{H}_i^+} \right)
  \subset \mathcal{H}_i^-.
\end{cases}
\end{equation}
If we replace $\mathcal{H}_i^\pm$ by $\mathcal{\tilde{H}}_i^\pm$, the inclusions become sharp equalities:
\begin{subequations}
\label{eq:sharp_ping_pong}
  \begin{equation}
  \label{eq:spp1}
    \gamma_i \left(\mathbb{R}^{d+1,d} \setminus \overline{\mathcal{\tilde{H}}_i^-}\right) = \mathcal{\tilde{H}}_i^+;
  \end{equation}
  \begin{equation}
  \label{eq:spp2}
    \gamma_i^{-1} \left(\mathbb{R}^{d+1,d} \setminus \overline{\mathcal{\tilde{H}}_i^+}\right) = \mathcal{\tilde{H}}_i^-.
  \end{equation}
\end{subequations}
Indeed \eqref{eq:spp1} is true by definition of $\mathcal{\tilde{H}}_i^+$, and \eqref{eq:spp2} follows from \eqref{eq:spp1} because $\gamma_i$ is continuous and $\mathcal{H}_i^-$ is a regular domain.

In order to have a real "ping-pong configuration", we also need to check that the sets $\mathcal{\tilde{H}}_i^\pm$ are pairwise disjoint. But using their definition and \eqref{eq:aff_ping_pong}, we know that they are included in the sets $\mathcal{\tilde{H}}_i^\pm$, which are pairwise disjoint by hypothesis (${\bf t} \in {\bf T}$).
\begin{enumerate}[(i)]
\item is now trivial. Indeed, we see by induction that for every $\gamma = \gamma_{i_1}^{\sigma_1} \ldots \gamma_{i_k}^{\sigma_k} \in \Gamma$, the image $\gamma(\mathcal{H}^0)$ lies in $\mathcal{\tilde{H}}_{i_1}^{\sigma_1}$, which is by definition disjoint from $\mathcal{H}^0$.

\item is the hard part. The particular case $d=1$ was done by Drumm in \cite{Dru93} (proof of Theorem 4); see also \cite{CG00}, section 4. Our proof is closely analogous.

From now on, we work in metric given by $N_0$.

Before proceeding, we need a small geometric lemma.
\end{enumerate}
\end{proof}

\begin{lemma} \mbox{ }
\label{angle_control}
\begin{enumerate}[(i)]
\item Let $V$ and $W$ be two MTIS'es. Then we have
\[\alpha(V_Q^\perp, W_Q^\perp) = \alpha(V, W)\]
(where $V_Q^\perp$ means the space $Q$-orthogonal to $V$).
\item Let $V'$ be a space $Q$-orthogonal to some MTIS $V$ and $x \in S$. Then we have
\[\sin \alpha(x,\; V' \cap S) = \sqrt{2} \sin \alpha(x,\; V').\]
\end{enumerate}
\end{lemma}

\begin{proof} \mbox{ }
\begin{enumerate}[(i)]
\item It is obvious that $\alpha(V_{N_0}^\perp, W_{N_0}^\perp) = \alpha(V, W)$ (where $V_{N_0}^\perp$ means, similarly, the space $N_0$-orthogonal to $V$). On the other hand, for every MTIS $V$, we have $V_Q^\perp = \varsigma(V_{N_0}^\perp)$, where the map $\varsigma := \Id_S \oplus (-\Id_T)$ is an $N_0$-isometry. The required equality follows.
\item First note that any plane contained in $V'$ intersects $S$ at an angle of $\frac{\pi}{4}$. Indeed, such a plane contains vectors $x$ such that $Q(x, x) = 0$ (because it intersects $V$), but no vectors $x$ such that $Q(x, x) < 0$. On the other hand, it follows from the definition of $N_0$ and $Q$ that $\alpha(x, S) <$ (resp. $=$, $>$) $\frac{\pi}{4}$ iff $Q(x, x) >$ (resp. $=$, $<$) $0$.

Now let $X$ be the 3-space spanned by $x$, $\pi_{V'}(x)$ and the line $V' \cap S$ (here $\pi_{V'}$ stands for the $N_0$-orthogonal projection onto $V'$). We know that the planes $V' \cap X$ and $S \cap X$ intersect at an angle of $\frac{\pi}{4}$, and that the plane spanned by $x$ and $\pi_{V'}(x)$ is perpendicular to $V' \cap X$. The spherical version of the law of sines yields the identity $\sin \alpha(x, V' \cap S) = \sqrt{2} \sin \alpha(x, \pi_{V'}(x))$, QED.
\qed
\end{enumerate}
\end{proof}

\begin{proof}[of Proposition \ref{fundamental_region}, continued]
We proceed by contradiction: let $x_0 \in \mathbb{R}^{d+1,d}$ such that
\begin{equation}
\label{eq:x_condition}
\forall \gamma \in \Gamma,\quad \gamma(x_0) \not\in \overline{\mathcal{H}^0}.
\end{equation}
Then there is a (unique) sequence $(i_k, \sigma_k)$ (indexed by $k \geq 1$) of elements of $\{1, \ldots, n\} \times \{-, +\}$, such that for all $k \geq 0$, we have
\[\gamma^{-[k]}(x_0)
    \in \mathcal{\tilde{H}}_{i_{k+1}}^{\sigma_{k+1}},\]
where $\gamma^{[k]} := \gamma_{i_1}^{\sigma_1} \ldots \gamma_{i_k}^{\sigma_k}$, and $\gamma^{-[k]}$ is shorthand for $(\gamma^{[k]})^{-1}$. Indeed, by induction, suppose that we have constructed the first $k$ terms (for some $k \geq 0$); then we have, by hypothesis
\[\gamma^{-[k]}(x_0)
    \;\in\; \mathbb{R}^{d+1,d} \setminus \overline{\mathcal{H}^0}
    \;=\; \bigcup_{i=1}^n \bigcup_{\sigma = \pm} \mathcal{\tilde{H}}_i^\sigma,\]
which allows us to pick an appropriate pair $(i_{k+1}, \sigma_{k+1})$ (which is actually unique, since the $\mathcal{\tilde{H}}_i^\pm$ are disjoint).

Note also that the word $\gamma^{[k]}$ is always reduced, \ie for all $k \geq 1$, we have $(i_{k+1}, \sigma_{k+1}) \neq (i_k, -\sigma_k)$. Indeed, we have $\gamma^{-[k-1]}(x_0) \in \mathcal{\tilde{H}}_{i_k}^{\sigma_k}$; since $\gamma_{i_k}^{\sigma_k}$ is bijective, applying \eqref{eq:spp1} (assuming $\sigma_k = +1$; otherwise \eqref{eq:spp2}), we get
\begin{align*}
\gamma^{-[k]}(x_0) &=   \gamma_{i_k}^{-\sigma_k} \left( \gamma^{-[k-1]}(x_0) \right) \\
                   &\in \mathbb{R}^{d+1,d} \setminus \overline{\mathcal{\tilde{H}}_{i_k}^{-\sigma_k}}.
\end{align*}

We may also suppose that for infinitely many values of $k$, the word $\gamma^{[k]}$ is cyclically reduced (in other terms, $(i_k, \sigma_k) \neq (i_1, -\sigma_1)$.) Indeed, otherwise, we may replace $x_0$ by $\gamma_i^\sigma(x_0)$, where $(i, \sigma)$ is a pair such that the set of indices $k$ such that $(i_k, \sigma_k) \neq (i, -\sigma)$ is infinite and also contains 1 (such a pair always exists). Then the new value still satisfies \eqref{eq:x_condition}, and the sequence $(i_k, \sigma_k)$ changes by appending $(i, \sigma)$ at the beginning.

Without loss of generality, let us suppose that $(i_1, \sigma_1) = (1, +)$.

Now an easy induction shows that the following domains form a decreasing sequence that concentrates on $x_0$ :
\begin{equation}
\label{eq:decreasing_sequence}
\mathcal{\tilde{H}}_1^+
\supset \gamma^{[1]}(\mathcal{\tilde{H}}_{i_2}^{\sigma_2})
\supset \gamma^{[2]}(\mathcal{\tilde{H}}_{i_3}^{\sigma_3})
\supset \ldots
\ni x_0.
\end{equation}

Next, we define
\[\Delta := S \cap V_{1, \subg}^{\halfperp}\]
(recall that $S$ is a maximal positive definite space). We know that $S$ is a $(d+1)$-dimensional space, $V_{1, \subg}^{\halfperp}$ is half a $(d+1)$-dimensional space and $S \cap V_{1, \subg} = 0$ (since $S$ is positive definite and $V_{1, \subg}$ is isotropic). Thus $\Delta$ is a half-line. We also define $P$ to be the ($d$-dimensional) hyperplane of $S$ that is $N_S$-orthogonal (or, equivalently, $Q$-orthogonal, or also $N_0$-orthogonal) to $\Delta$. To avoid cumbersome periphrases, in the following, we shall often use terms such as "above" and "below", having in mind that "up" is the direction where $\Delta$ points.

Now consider the set $\mathcal{\tilde{H}}_1^+ \cap (x_0 + S)$ (here $(x_0 + S)$ stands for the affine space passing through $x_0$ and parallel to $S$). It is contained in an affine half-space of $(x_0 + S)$ lying above a hyperplane parallel to $P$. Indeed, from \eqref{eq:eps_i_and_eps} it follows that
\[\mathcal{H}_{\mathbb{S},1}^+ \subset B_{N_0} \left( \pi_\mathbb{S}(V_{1, \subg}^\halfperp), \frac{\eps}{3} \right).\]
Without loss of generality, we may suppose that $\eps \leq \diam \mathbb{P}(\mathbb{R}^{d+1, d}) = \frac{\pi}{2}$ (indeed the separation of no frame or frameset may exceed that value). Then the radius of the right-hand-side neighborhood is no larger than $\frac{\pi}{6}$. It follows from Lemma \ref{angle_control} that
\[B_{N_0} \left( \pi_\mathbb{S}(V_{1, \subg}^\halfperp), \frac{\pi}{6} \right) \cap S
  \;\;\subset\; B_{N_0}\left( \pi_\mathbb{S}(V_{1, \subg}^\halfperp \cap S), \frac{\pi}{4} \right)\]
(the angle $\frac{\pi}{4}$ is the solution to $\sin x = \sqrt{2} \sin \frac{\pi}{6}$). The desired property may be deduced from here.

Applying \eqref{eq:decreasing_sequence}, we see that for every $k \geq 0$, the domain
\[\gamma^{[k]}(\mathcal{\tilde{H}}_{i_{k+1}}^{\sigma_{k+1}}) \cap (x_0 + S)\]
is included in a half-space of $(x_0 + S)$ lying above a hyperplane parallel to $P$. We define $P_k$ to be the uppermost such hyperplane; we call $a_k$ the intersection of $P_k$ with the line containing $(x_0 + \Delta)$, and we set, for $k \geq 1$, $\delta_k = a_k - a_{k-1}$.

The result now follows from:

\begin{lemma}
\label{gap_bounded_below}
There is a constant $\delta_{\min} > 0$ such that for every $k \geq 1$, whenever $(i_k, \sigma_k) \neq (1, -)$, $\|\delta_k\|_{N_0} \geq \delta_{\min}$.
\end{lemma}

Indeed, from \eqref{eq:decreasing_sequence}, it follows that the sequence $(a_k)$ is increasing and bounded above by $x_0$. However, we have chosen $x_0$ in such a way that the condition of Lemma \ref{gap_bounded_below} occurs infinitely often. It follows that $(a_k)$ is unbounded, which is a contradiction.
\qed
\end{proof}

\begin{proof}[of Lemma \ref{gap_bounded_below}]
We still work in metric given by $N_0$.

Let $k \geq 1$ be an index such that $(i_k, \sigma_k) \neq (1, -)$, so that $g^{[k]}$ is cyclically reduced. We know that the group $G$ is pseudohyperbolic; by Proposition \ref{product_pseudohyperbolic}, we have
\[\alpha(V_{\subg}(g^{[k]}), V_{1, \subg}) \ll_\eps s(G).\]
As $s(G) \leq s_3(\eps)$, by choosing $s_3(\eps)$ small enough, we may suppose that this angle is no larger than $\frac{\pi}{6}$. By Lemma \ref{angle_control}, it follows that
\[\alpha(V_{\subge}(g^{[k]}) \cap S,\; V_{1, \subge} \cap S) \leq \frac{\pi}{4}\]
(remember that $V_{1, \subge} \cap S$ is the line containing $\Delta$).

Now let $\eta_k$ be the projection of $\delta_k$ onto $V_{\subge}(g^{[k]}) \cap S$ parallel to $P$. Then we have
\begin{equation}
\label{eq:eta_delta}
\|\delta_k\| \geq \left(\cos \frac{\pi}{4}\right) \|\eta_k\| = \frac{\sqrt{2}}{2}\|\eta_k\|.
\end{equation}

Next, still by Proposition \ref{product_pseudohyperbolic}, we know that $g^{[k]}$ is pseudohyperbolic, $\frac{\eps}{3}$-separated and 1-contracting; let $\mathcal{V}^{[k]}$ be its frame. By definition, the norm of $g^{[k]}$ restricted to $V_{\subg}(g^{[k]})$ (resp. $V_{\sube}(g^{[k]})$) is equal to $s(g^{[k]})$ (resp. $1$). It follows that
\begin{align*}
\left\| \restr{g^{-[k]}}{V_{\subge}(g^{[k]})} \right\|_{N_{\mathcal{V}^{[k]}}}
  &= \max \left(\left\| \restr{g^{-[k]}}{V_{\subg}(g^{[k]})} \right\|,\;
                \left\| \restr{g^{-[k]}}{V_{\sube}(g^{[k]})} \right\|\right) \\
  &= \max \left(\left\| (g^{[k]}_{\subg})^{-1} \right\|,\; 1 \right) \\
  &= 1.
\end{align*}
From Lemma \ref{uniformly_equivalent}, we deduce
\[\left\| \restr{g^{-[k]}}{V_{\subge}(g^{[k]})} \right\|_{N_0} \ll_\eps 1;\]
given that, by construction, $\eta_k \in V_{\subge}(g^{[k]})$, we get
\begin{equation}
\label{eq:eta_eta}
\|\eta_k\| \gg_\eps \|g^{-[k]}(\eta_k)\|.
\end{equation}

Finally, let $x_k$ be any point that lies both in $P_k$ and $\overline{\gamma^{[k]}(\mathcal{\tilde{H}}_{i_{k+1}}^{\sigma_{k+1}})}$ (the intersection is nonempty by definition of $P_k$). Set $y_{k-1} := x_k - \eta_k$. Since the orthogonal projection of $\eta_k$ onto $\Delta$ is equal to $\delta_k$, it follows that $y_{k-1} \in P_{k-1}$, and in particular $y_{k-1} \not\in \gamma^{[k-1]}(\mathcal{\tilde{H}}_{i_k}^{\sigma_k})$. Applying $\gamma^{-[k]}$, we get
\[\begin{cases}
\gamma^{-[k]}(x_k)     \in \overline{\mathcal{\tilde{H}}_{i_{k+1}}^{\sigma_{k+1}}} \\
\gamma^{-[k]}(y_{k-1}) \in \overline{\mathcal{\tilde{H}}_{i_k}^{-\sigma_k}}.
\end{cases}\]
Since $(i_{k+1}, \sigma_{k+1}) \neq (i_k, -\sigma_k)$, we have
\begin{equation}
\label{eq:eta_bound}
\|g^{-[k]}(\eta_k)\| = \|\gamma^{-[k]}(x_k) - \gamma^{-[k]}(y_{k-1})\| \geq d_{\min},
\end{equation}
where $d_{\min}$ is the smallest distance between any of the $\mathcal{\tilde{H}}_i^\sigma$ (which is nonzero as shown in the proof of Lemma \ref{regions_disjoint_open}).

Joining \eqref{eq:eta_delta}, \eqref{eq:eta_eta} and \eqref{eq:eta_bound} together, we get indeed a lower bound for $\|\delta_k\|$ that does not depend on $k$.
\qed
\end{proof}

\begin{acknowledgements}
I would like to thank my PhD advisor, Mr. Yves Benoist, whose help while I was working on this paper has been invaluable to me.
\end{acknowledgements}

\end{document}